\documentclass[12pt,reqno]{amsart}
\usepackage[dvipsnames]{xcolor}
\usepackage{graphicx}
\usepackage{tikz-cd}
\usepackage{amsmath, amssymb}
\usepackage{mathrsfs}
\usepackage{amsthm}
\usepackage{multicol}
\usepackage{color}
\usepackage{bm}
\usepackage{mathtools}
\usepackage{nameref}
\usepackage{url}
\usepackage[colorlinks=true,linkcolor=blue!60!black,citecolor=teal!80!black,urlcolor=RoyalBlue]{hyperref}
\usepackage[utf8x]{inputenc} 
\usepackage[left=2.7cm,right=2.7cm,top=2.65cm,bottom=2.65cm]{geometry}
\usepackage{enumitem}
\usepackage{accents}
\usepackage{import}

\usepackage{faktor}

\newtheorem{thm}{Theorem}[section]

\newtheorem{prop}[thm]{Proposition}

\newtheorem{dfn&prop}[thm]{Definition and Proposition}
\newtheorem{dfn&thm}[thm]{Definition and Theorem}

\theoremstyle{definition}
\newtheorem{defn}[thm]{Definition}
\newtheorem{observation}[thm]{Observation}

\newtheorem*{structure}{Structure of the article}

\newtheorem{example}[thm]{Example}
\newtheorem{discussion}[thm]{}

\theoremstyle{remark}
\newtheorem*{Claim}{Claim}

\newtheorem*{remark}{Remark}
\newtheorem*{notations}{Notation}

\numberwithin{equation}{section}

\newenvironment{subproof}{\begin{proof}[Proof of claim.]}{%
\end{proof}}

\def\phi{\varphi}
\newcommand{\itin}{\operatorname{itin}}
\def\F{{\mathcal{F}}}

\def\M{{\mathcal{M}}}
\def\Exad{\mathscr{S}^{\mathbb{N}}}
\def\Exadpm{\mathscr{S}^{\mathbb{N}}_\pm}
\def\B{{\mathcal{B}}}
\def\CB{{\mathcal{CB}}}

\def\C{{\mathbb{C}}}
\def\D{{\mathbb{D}}}
\def\N{{\mathbb{N}}}
\def\Z{{\mathbb{Z}}}
\def\R{{\mathbb{R}}}

\def\Q{\mathbb {Q}}

\def\Or{{\mathcal{O}}}
\def\Ort{{\widetilde{\mathcal{O}}}}

\newcommand{\Ima}{\operatorname{Im}}
\newcommand{\Rea}{\operatorname{Re}}

\newcommand{\Crit}{\operatorname{Crit}}

\newcommand{\Addr}{\operatorname{Addr}}
\newcommand{\ul}{\underline}
\newcommand{\cl}{\overline}

\newcommand{\ultau}{\underline{\tau}}
\newcommand{\T}{\mathcal{T}}
\newcommand{\V}{\mathcal{V}}
\newcommand{\AV}{\operatorname{AV}}
\newcommand{\CV}{\operatorname{CV}}
\newcommand{\Orb}{\operatorname{Orb}}

\def\s{{\underline s}}

\newcommand*{\defeq}{\mathrel{\vcenter{\baselineskip0.5ex \lineskiplimit0pt
			\hbox{\scriptsize.}\hbox{\scriptsize.}}}%
	=}

\newcommand{\eqdef}{=\mathrel{\vcenter{\baselineskip0.5ex \lineskiplimit0pt
			\hbox{\scriptsize.}\hbox{\scriptsize.}}}}

\title{Topological dynamics of cosine maps}

\author[L. Pardo-Sim\'{o}n]{Leticia Pardo-Sim\'{o}n}

\address{Department of Mathematics \\ The University of Manchester \\ Manchester \\ M13 9PL \\ United Kingdom \\ 
	 \textsc{\newline \indent 
	   \href{https://orcid.org/0000-0003-4039-5556%
	     }{\includegraphics[width=1em,height=1em]{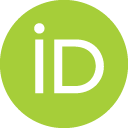} {\normalfont https://orcid.org/0000-0003-4039-5556}}
	       }}
\email{leticia.pardosimon@manchester.ac.uk}
	
\begin{document}

\begin{abstract} 
The set of points that escape to infinity under iteration of a cosine map, that is, of the form $C_{a,b} \colon z \mapsto ae^z+be^{-z}$ for $a,b\in \C^\ast$, consists of a collection of injective curves, called \textit{dynamic rays}. If a critical value of $C_{a,b}$ escapes to infinity, then some of its dynamic rays overlap pairwise and \textit{split} at critical points. We consider a large subclass of cosine maps with escaping critical values, including the map $z\mapsto \cosh(z)$. We provide an explicit topological model for their dynamics on their Julia sets. We do so by first providing a model for the dynamics near infinity of any cosine map, and then modifying it to reflect the splitting of rays for functions of the subclass we study. As an application, we give an explicit combinatorial description of the overlap occurring between the dynamic rays of $z\mapsto \cosh(z)$, and conclude that no two of its dynamic rays land together. 
\end{abstract}
\maketitle

\section{Introduction}\label{sec_intro}

The theory of iteration of transcendental entire functions  $f \colon \C \rightarrow \C$ dates back to Fatou's seminal work from 1926; \cite{Fatou}. The locus of stable behaviour of such a function – more precisely, the set of $z\in \C$ at which the family $\{f^n\}_{n\in\N}$ is equicontinuous with respect to the spherical metric – is today called  the \textit{Fatou set} $F(f)$, while its complement $J(f)\defeq  \C \setminus F(f)$ is the \textit{Julia set}. We are also interested in the \textit{escaping set}
\begin{equation*}
I(f)\defeq \{z\in \C \colon f^n(z)\to \infty \text{ as } n\to \infty\},
\end{equation*}
as $J(f)=\partial I(f)$ and, in the cases we will consider, $J(f)=\overline{I(f)}$; \cite{eremenkoclassB}.

Fatou observed that the Julia sets of certain sine functions contain infinitely many curves to infinity. In the 1980s, Devaney, with a number of co-authors, studied this phenomenon further, showing the existence of such curves for some functions in the \textit{exponential family} $E_\kappa \colon z \mapsto e^z+ \kappa$, $\kappa \in \C \setminus \{0\}$. These curves, that escape uniformly to infinity, are now known as \textit{(Devaney) hairs} or \textit{dynamic rays}, and they provide a foliation of the escaping set $I(E_\kappa)$ for all $\kappa$; \cite{Schleicher_Zimmer_exp}. See the formal definition of \textit{dynamic ray} in Definition~\ref{def_ray}.

If the asymptotic value of $E_\kappa$, i.e., its parameter $\kappa$, converges to an attracting or parabolic cycle, then all dynamic rays of $E_\kappa$ \textit{land}, that is, have a unique finite accumulation point; \cite{lasseTopExp}. This provides a total description of $J(E_\kappa)$ as a collection of unbounded escaping curves together with their landing points; compare to \cite{AartsOversteegen, mashael_thesis} for further topological characterizations. However, there is no such complete description of $J(E_\kappa)$ when $\kappa$ escapes to infinity. In contrast, Rempe showed in \cite{lasse_nonlanding} that for escaping parameters, the accumulation sets of uncountably many dynamic rays of $E_\kappa$ are indecomposable continua containing the rays themselves. 

The existence of dynamic rays is now known for a much larger class of functions  belonging to the  \textit{Eremenko-Lyubich class $\B$}, which consists of all transcendental entire functions with bounded singular set. Recall that for an entire map $f$, its \textit{singular set} $S(f)$ is the closure of the set of its asymptotic and critical values. More precisely, it is shown in \cite{RRRS} that for certain $f\in \B$, including all those of finite order, some iterate of every $z\in I(f)$ can be connected to infinity by a dynamic ray. As in the exponential case, the landing behaviour of rays for such an $f$ depends on the behaviour under iteration of their singular set, that is, on their \textit{postsingular set} $P(f)\defeq \overline{\bigcup^\infty_{n=0} f^n(S(f))}$. Note that for $f$ a postsingularly bounded entire function, $P(f)$ is nicely separated from infinity, where dynamic rays start. This has played a crucial role when proving that all dynamic rays of certain $f \in \B$ with bounded postsingular set land, e.g., \cite{Schleicher_entire, lasseRigidity, helenaSemi,mashael}. 
 
In this paper we are interested in the case when critical values escape. Note that  the singular set of any map in the \textit{cosine family}, that is, of the form \[C_{a,b} \colon z\mapsto ae^z+be^{-z} \quad \text{ for }a,b\in \C\setminus \{0\},\] 
consists of two critical values, with their preimages being critical points of local degree~$2$. The explicit nature of this family will allow us to obtain a complete description of the Julia set of certain cosine maps with escaping critical values. But first, we note that, as in the exponential case, when both critical values of $C_{a,b}$ belong to an attracting basin all dynamic rays land \cite{RRRS}, and the same holds when $P(C_{a,b})$ is strictly preperiodic \cite{dierkParadox}.

If, instead, a critical value of $C_{a,b}$ escapes to infinity, dynamic rays \textit{split}  at critical points, and the structure of $J(C_{a,b})$ is much more complicated. To illustrate this, consider the map $ C_{1/2,1/2} \colon z\mapsto \cosh(z)$, whose critical values $-1$ and $1$ escape to infinity in $\R^+$. Note that $0$ is a critical point, and it is easy to check that $(-\infty, 0]$ and $[0, \infty)$ are pieces of dynamic rays. The vertical segments $[0, -i\pi /2]$ and $[0, i\pi /2]$ are mapped univalently to $[0, 1] \subset \R^+$, and thus, the union of each segment with either  $(-\infty, 0]$ or $[0, \infty)$ forms a different piece of ray. This structure can be interpreted as four pieces of rays that partially overlap pairwise. Their endpoints $-i\pi /2$ and $i\pi/2$ are preimages of $0$, and so the structure described has a preimage attached to each of them; see Figure~\ref{fig:rays_cosh}. This leads again to two possible extensions of each of them. Our results will imply that in this case, further extensions can be performed in a systematic fashion that converges to four dynamic rays that land.

Analogous results will be achieved for those cosine maps whose singular orbits are ``sufficiently spread'' in the following sense: 
\begin{defn}\label{def_sps_intro} A cosine map $f$ is \emph{strongly postcritically separated (sps)} if $P(f)\cap F(f)$ is compact and there exists $\epsilon>0$ such that for all distinct $z,w\in P(f)\cap J(f)$, $\vert z-w\vert\geq \epsilon \max\{\vert z \vert, \vert w \vert\}$.
\end{defn}

\begin{remark}
A more general notion of strongly postcritically separated maps is introduced in \cite{mio_orbifolds}, see Definition \ref{def_strongps}, where it is shown that they expand a suitable \textit{orbifold metric} in a neighbourhood of their Julia sets. This is key to our results.
\end{remark}

The following theorem shows how for cosine and exponential maps with escaping singular values, their different nature, being critical rather than asymptotic values, changes drastically the topology of their respective Julia sets. 
\begin{thm} \label{thm_landing} Let $f$ be a strongly postcritically separated cosine map. Then, every dynamic ray of $f$ lands, and every point in $J(f)$ is either on a dynamic ray or it is the landing point of at least one such ray.
\end{thm}

We remark that Theorem \ref{thm_landing} will follow from our results but it is not new, as it is also a consequence of \cite[Theorem~1.2]{mio_splitting}. More specifically, in \cite{mio_splitting} the same result is obtained for a more general class of functions in $\B$ with dynamic rays. In turn, that result is a consequence of a stronger one: \cite[Theorem 1.4]{mio_splitting} provides an abstract topological model for the action of any such $f$ on its Julia set, a model that is based on the dynamics of an entire map $g$ on its \textit{parameter space}, i.e., quasiconformally equivalent to $f$. More precisely, in order to reflect the splitting of rays as described for $z\mapsto \cosh(z)$, two copies of $J(g)$ are considered as a model space, namely $J(g)_\pm\defeq J(g)\times \{-,+\}$, with some special topology that preserves the order of rays at infinity. Then, the model function $\tilde{g}\colon J(g)_\pm \to J(g)_\pm$, acting as $g$ on the first coordinate and as the identity on the second, is shown to be semiconjugate to $f\vert_{J(f)}$.

As the main result of this paper, using the general framework that  \cite{mio_splitting} provides, we construct in Theorem \ref{thm_strong_cosine_intro} a simpler topological model for the action of any sps cosine map on its Julia set. Before we give more details of our model, we highlight the main advantages that it presents over the model from \cite{mio_splitting}:
\begin{itemize}
\item Since the model function in \cite{mio_splitting} acts on its first coordinate as the restriction of an entire function to its Julia set, its dynamics are still very complicated. Instead, the dynamics of our model function will be much simpler, coding the exponential growth of real parts of points under cosine maps, and the location of their imaginary parts with respect to a Markov-type partition of the plane.
\item The explicit nature of the maps considered allows us to provide sharper results: in \S 5 we improve our model and provide a complete description of the topological dynamics of $z\mapsto \cosh(z)$. In particular, we are able to conclude that no two of its dynamic rays land together. 
\end{itemize}

As a first step on the construction of our model for sps cosine maps, we provide a model that relates to the dynamics near infinity of all cosine maps. We note that the idea of relating the dynamics of one map to those of a simpler one has been successfully exploited in the polynomial case using Böttcher's Theorem; see Douady's Pinched Disk model \cite{douady_pinchedmodel}. However, due to the essential singularity at infinity, for transcendental maps Böttcher's Theorem no longer applies, and so our techniques are different. 
Given that cosine maps act like the exponential map, up to a constant factor, in left and right half-planes sufficiently far away from the imaginary axis, we start by constructing a \textit{topological model for escaping cosine dynamics} inspired by Rempe's model for exponential maps \cite{lasseTopExp}. Roughly speaking, this model is formed by a topological space $J(\F)$ consisting of a collection of disjoint curves, and a continuous map $\F: J(\F)\rightarrow J(\F)$, where $\F$ codes the exponential growth of the real parts of points under cosine maps, as well as the orbits of their imaginary parts with respect to a Markov-type partition of the plane, see Definition~ \ref{defn_modelcosine}. In particular, its dynamics are straightforward.

The following theorem is particularly strong when $f$ is of \textit{disjoint type}, that is, when its Fatou set  is an attracting basin and $P(f)\subset F(f)$. For each $R>0$, we denote
\begin{equation}
J_R(f)\defeq \{z\in J(f)\colon \vert f^n(z)\vert \geq R \text{ for all } n\geq 1\}.
\end{equation}

\begin{thm}[Model for cosine dynamics] \label{thm_conjcosine_intro} Let $f$ be a cosine map and let $J(\F)$ and $\F$ be as defined as above. Then there exists a constant $R\geq 0$ and a continuous map $\Phi\colon J(\F )\rightarrow J(f)$ such that $\Phi \vert_{\Phi^{-1}(J_R(f))}$ is a homeomorphism, 
	$$ \Phi \circ \F=f \circ \Phi \quad \text{ on } \quad \Phi^{-1}(J_R(f))$$  and $\Phi(I(F)) \subset I(f)$. If in addition $f$ is of disjoint type, then  $\Phi\colon J(\F)\to J(f)$ is a homeomorphism and $\Phi(I(\F))=I(f).$
\end{thm}

Compare to \cite[Theorems 4.2 and 9.1]{lasseTopExp} for similar results on the exponential family. Note that any escaping point of any cosine map $f$ eventually enters $J_R(f)$ for every $R>0$, and so 
Theorem \ref{thm_conjcosine_intro} provides a model for its escaping dynamics.

However, we aspire to describe the dynamics of every cosine map in the whole of its Julia set. In particular, in the presence of escaping critical values, any model must reflect the  splitting of dynamic rays at (preimages of) critical points, as described for the map $z\mapsto \cosh(z)$ before. Note that our analysis on this map, where each ray tail ``splits into two'' at critical points, suggests considering two copies of each ray, and mapping each copy to one of the two possible extensions. With that aim, we define the model space for sps maps as $J(\F)_\pm\defeq J(\F)\times \{-,+\}$ with a topology that preserves the circular order of rays at infinity; see Definition \ref{defn_modelspace_cos}. The model map $\widetilde{\F}\colon J(\F)_\pm \rightarrow J(\F)_\pm$ is defined as $\F$ on the first coordinate and as the identity on the second. Then, our main result is as follows.
\begin{thm}[Model for the dynamics of strongly postcritically separated cosine maps]\label{thm_strong_cosine_intro} Let $f$ be in the cosine family and strongly postcritically separated. Then, there exists a continuous surjective function $\hat{\phi} \colon J(\F)_{\pm} \rightarrow J(f)$ so that $f\circ\hat{\phi} = \hat{\phi}\circ \widetilde{\F}$ on $J(\F)_{\pm}$.
\end{thm}
\noindent See Theorem \ref{thm_strong_cosine} for a more detailed version of Theorem \ref{thm_strong_cosine_intro}. In particular, Theorem \ref{thm_landing} will follow readily.

Some examples of sps cosine maps that have already appeared in the literature of holomorphic dynamics are $z\mapsto \cosh(z)$ and $z\mapsto \cosh^2(z)$, see \cite{bishop2015,david_Shishikura_Wandering,ripponFast}. In Section \ref{sec_cosh_indent} we improve Theorem \ref{thm_strong_cosine_intro} for these maps by modifying our model as to obtain a conjugacy. In particular, we provide an explicit combinatorial description of the overlap occurring between their dynamic rays  and conclude that for both functions, no two of their dynamic rays land together.

In order to achieve our results on the map $z\mapsto \cosh(z)$, we introduce the notion of \textit{itineraries} as sequences that encode the orbits of points on its Julia set with respect to a dynamical partition, an idea already used, for example, in \cite{dierkParadox, helena_thesis}. In Appendix \ref{sec_landing}, we extend this concept to a larger subclass  of functions in $\B$ with dynamic rays. Namely, to all strongly postcritically separated maps that belong to the class $\CB$, the latter including all functions that are a finite composition of class $\B$ functions of finite order. This new tool will allow us to provide in Theorem \ref{thm_landing_intro} criteria for their dynamic rays landing together.

\begin{structure}
	In Section \ref{sec_prelim_cosine} we review some basic properties of cosine dynamics and fix for each  cosine map a choice of \textit{external addresses}, a combinatorial tool used in the study of functions in $\B$. We define the model $(J(\F), \F)$ in Section  \ref{sec_model_disjoint}, study its properties and prove Theorem \ref{thm_conjcosine_intro}. Section  \ref{sec_model_sps} includes our results on sps cosine maps. Namely, we define the model $(J(\F)_\pm, \widetilde{\F})$ and prove Theorems \ref{thm_strong_cosine_intro} and  \ref{thm_landing}. Next, in Section  \ref{sec_cosh_indent} we sharpen our main result for the maps $z \mapsto \cosh(z)$ and $z \mapsto \cosh^2(z)$ and provide a detailed description of their topological dynamics. Finally, we provide combinatorial criteria for dynamic rays of sps functions in $\CB$ landing together in Appendix \ref{sec_landing}.
\end{structure}

\noindent\textbf{Basic notation.} As used throughout this section, the Fatou, Julia and escaping set of an entire function $f$ are denoted by $F(f)$, $J(f)$ and $I(f)$ respectively. The set of critical values is $\CV(f)$, that of asymptotic values is $\AV(f)$, and the set of critical points will be $\Crit(f)$. The set of singular values of $f$ is $S(f)$, and $P(f)$ denotes the postsingular set. Moreover, $P_{J}\defeq P(f)\cap J(f)$ and $P_{F}\defeq P(f)\cap F(f)$. A disc of radius $\epsilon$ centred at a point $p$ will be $\D_\epsilon(p)$, and $\C^\ast \defeq \C\setminus \{0\}$. We will indicate the closure of a domain $U$ by $\overline{U}$, that must be understood to be taken in $\C$. For a holomorphic function $f$ and a set $A$, $\Orb^{-}(A)$ and $\Orb^{+}(A)$ are the backward and forward orbit of $A$ under $f$, respectively. That is, $\Orb^{-}(A)\defeq \bigcup^{\infty}_{n=0} f^{-n}(A)$ and $\Orb^{+}(A)\defeq \bigcup^{\infty}_{n=0} f^{n}(A).$

\noindent\textbf{Acknowledgements.} I am very grateful to Lasse Rempe for his guidance and support, as well as for providing computer-generated pictures, see Figures 2-5. I also thank Dave Sixsmith, David Martí-Pete and the referee for much helpful feedback.

\section{Cosine dynamics and external addresses}\label{sec_prelim_cosine}
We start by revising basic properties of cosine maps. We refer to \cite{dierkCosine, dierkParadox} for extensive work on their dynamics. Note also that cosine maps arise as lifts of holomorphic self-maps of $\C^*$; see \cite[Corollary 1.5]{nuria_david}. Recall that for a holomorphic map $f:\widetilde{S}\rightarrow S$ 
between Riemann surfaces, the \emph{local degree} of $f$ 
at a point $z_0\in \widetilde{S}$, denoted by $\deg(f,z_0)$, is the unique integer $n\geq 1$ 
such that the local power series development of $f$ is of the form
\begin{equation*}
f(z)=f(z_0) + a_n (z-z_0)^n + \text{(higher terms)},
\end{equation*}
where $a_n\neq 0$. Thus, $z_0\in \C$ is a critical point of $f$ if and only if $\deg(f,z_0)>1$. We say that $f$ has \textit{bounded criticality} on a set $A$ if $\AV(f) \cap A=\emptyset$ and there exists a constant $M<\infty$ such that $\deg(f,z)<M \text{ for all } z \in A.$
\begin{discussion}[Basic properties of cosine maps]\label{dis_basic_cosine}
	Each cosine map $f(z)\defeq ae^z+be^{-z}$ with $a,b \in \C^\ast $ is $2\pi i$-periodic and has exactly two critical values, namely $\pm 2\sqrt{ab}$. Furthermore, any preimage of a critical value is a critical point of local degree $2$, and hence both critical values are \textit{totally ramified}. More specifically, 
	$$\Crit(f)=\left\{\frac{1}{2}\log\left(\frac{a}{b}\right)+ \pi i n : n\in \Z \right\},$$
	where the branch of the logarithm is chosen such that $\vert \Ima(\frac{1}{2}\log(\frac{a}{b}))\vert\leq \pi/2$. 
	It is easy to check that $f$ has no asymptotic values, and thus, $S(f)\eqdef\{v_1, v_2\}$, with $v_i=\pm 2\sqrt{ab}$ and signs chosen so that $v_1$ is the image of $\frac{1}{2}\log(\frac{a}{b})+2\pi i\Z$, while $v_2$ is the image of $\frac{1}{2}\log(\frac{a}{b})+\pi i+2\pi i\Z$. In particular, since $S(f)$ is bounded and $f$ is of order of growth one, the Julia set of any disjoint type cosine map is a \textit{Cantor bouquet}. Roughly speaking, a Cantor bouquet consists of an uncountable collection of curves to infinity satisfying a certain density condition; see \cite[Definition 2.1]{lasseBrushing}. Moreover, by the Denjoy-Carleman-Ahlfors theorem, for any choice of bounded domain $D\supset S(f)$, the number of connected components of $f^{-1}(\C \setminus D)$, called \textit{tracts}, is at most two. Note that for any such domain $D$, $f$ maps points for which the absolute value of their real part is sufficiently large, to $\C \setminus D$. Hence, a left and a right half plane are contained in the union of tracts, which implies that $f$ has at least two, and hence exactly two, tracts for any choice of $D$.
\end{discussion}
\begin{observation}[Parameter space of cosine maps]\label{obs_spacecosine}
All cosine maps belong to the same parameter space; that is, any two cosine maps are quasiconformally equivalent. To see this, let $f(z)\defeq ae^z+be^{-z}$ and $g(z)\defeq ce^z+de^{-z}$ for $a,b,c,d \in \C^\ast$. Consider the linear maps $\psi(z)\defeq z+ \log \sqrt{\frac{bc}{ad}}$ and $\phi(z)\defeq \sqrt{\frac{bc}{ad}} z$, which are clearly quasiconformal. Then, for all $z\in \C$,
	$$ (f\circ \psi)(z)=a e^z\sqrt{\frac{bc}{ad}}+be^{-z} \sqrt{\frac{ad}{bc}}=c e^z\sqrt{\frac{ab}{cd}}+de^{-z} \sqrt{\frac{ab}{cd}}=(\phi\circ g)(z).$$ 
\end{observation}

Consequently, in order to prove the second part of Theorem \ref{thm_conjcosine_intro}, by \cite[Theorem~3.1]{lasseRigidity}, it suffices to construct a conjugacy between $\F\colon J(\F)\rightarrow J(\F)$ and any specific disjoint type cosine map. We could have followed this approach, but for the sake of generality, our proof will relate to any disjoint-type cosine map. 

Let us fix a cosine function $f(z)\defeq ae^z+be^{-z}$ with $a,b\in \C^\ast$. We will make use of some estimates from \cite{dierkCosine} that require of constants related to the following:
	\begin{equation*}
	\mathcal{K}(f)\!\defeq \!\max \! \left\{\!\left(\sqrt{\left\vert \frac{2b}{a} \right\vert} +\!\sqrt{\left\vert\frac{2a}{b} \right\vert}\right)\!(\vert a \vert + \vert b \vert), 8\vert ab \vert, 1, \frac{1}{2}\ln \left\vert \frac{2b}{a}\right\vert , \frac{1}{2}\ln\left\vert \frac{2a}{b}\right\vert, \ln \frac{16}{\vert a b \vert} \right\}\!.
	\end{equation*} 

More precisely, recall from \ref{dis_basic_cosine} that for any cosine map $f$ and for any choice of Jordan domain $D\supset S(f)$, $f^{-1}(\C\setminus D)$ has two connected components, that is, two tracts. We will choose $D$ large enough as to guarantee \textit{Euclidean expansion} within tracts, that is, that the modulus of the derivative of any point in the tracts is large enough. 
Let us choose $R>\mathcal{K}(f)$ big enough such that 
$$ \bigcup_{k\in \Z} \{z+2\pi k i \colon z \in \D_R \} \supset \{z \in \C \colon \vert \Rea z \vert\leq \mathcal{K}(f)\} $$ 
and $f(\D_R)\supset \D_{\mathcal{K}(f)}$. In particular, the domain $D\defeq f(\D_R)$ contains $S(f)$, and so we can define a pair of tracts $\T_f$ as the connected components of $f^{-1}(\C\setminus D)$. Then, by definition, 
\begin{equation}\label{eq_exp_tracts}
	\T_f \subseteq \{z\in \C \colon \vert \Rea (z)\vert >\mathcal{K}(f) \} \quad \text{ and } \quad S(f) \subset \D_{\mathcal{K}(f)} \subset \C \setminus f(\T_f).
\end{equation}
A simple calculation shows that for any  $z\in \C$ such that  $\vert  \Rea z\vert >\mathcal{K}(f)$,
\begin{equation}\label{eq_expansion}
\vert f' (z) \vert >2;
\end{equation}
see \cite[Lemma~3.6]{dierkCosine}. Hence, we say that $\T_f$ are \textit{expansion tracts}.

\begin{discussion}[Fundamental domains and inverse branches]\label{dis_inversenormalized} Let us fix a cosine map $f(z)\defeq ae^z+be^{-z}$ for some $a,b\in \C^\ast$, and let $\T_f$ be a pair of expansion tracts. Let $S(f)\eqdef\{v_1, v_2\}$, with $v_1$ and $v_2$ labelled according to \ref{dis_basic_cosine}. In particular, by  \eqref{eq_exp_tracts},  $S(f)\subset D\defeq \C \setminus f(\T_f)$ and $\{z\in \C \colon \vert \Rea (z)\vert <\max(\vert \Rea v_1\vert , \vert \Rea v_2 \vert )\}\subset \C\setminus \T_f$. If $\Ima(v_1)> \Ima(v_2)$, we define $\delta$ as the vertical straight line starting at $v_1$ in upwards direction restricted to $\C\setminus D$. If on the contrary $\Ima(v_1)< \Ima(v_2)$, $\delta$ is the downwards vertical line joining $v_2$ to infinity restricted to $\C\setminus D$. In any case, $\delta \subset \C\setminus (\T_f\cup D)$, and so we can define \textit{fundamental domains} for $f$ as the connected components of $\T_f\setminus f^{-1}(\delta)$; see \cite[\S 2]{mio_signed_addr} for the definition of fundamental domains in a more general setting. Since $f$ is in the cosine family, by definition, the image of points in $\R$ whose modulus is large enough have uniformly large modulus, and so they must be totally contained in a fundamental domain. By $2\pi i$-periodicity of $f$, the same holds for all their $2\pi i$-translates. Hence, for each $n\in \Z$, we denote by $F_{(n,R)}$ the fundamental domain that contains an unbounded subset of $2\pi n i + \R^+$, and by $F_{(n,L)}$ the fundamental domain that contains an unbounded subset of $2\pi ni + \R^-$. 
	Since $f$ maps each fundamental domain to its image $f(\T_f) \setminus \delta$ as a conformal isomorphism, see \cite[Proposition 2.19]{mio_thesis}, we can define for each $(n, \ast) \in \Z\times\{L, R\}$ the inverse branch
	\begin{equation}\label{eq_inversecosine}
	f^{-1}_{(n, \ast)}\colon f(\T_f)\setminus \delta \rightarrow F_{(n, \ast)}, 
	\end{equation}
	which in particular is a bijection.
\end{discussion}

\begin{observation}[Horizontal straight lines contained in fundamental domains]\label{obs_constantA} Following \ref{dis_inversenormalized}, by construction, there is a constant $A>\mathcal{K}(f)$ so that for all $n\in \Z$,
	\begin{align*}
	\{z: \Rea z<- A \text{ and } \Ima z=2\pi n\} &\subset F_{(n, L)}\qquad\text{and} \\
	\{z: \Rea z>A \text{ and } \Ima z=2\pi n\} &\subset F_{(n, R)}. 
	\end{align*}
\end{observation}

We note that our choice of fundamental domains in \ref{dis_inversenormalized} agrees with the partition defined in \cite[Sections 1 and 2]{dierkCosine}, where the maps ``$f^{-1}_{(n, \ast)}$'' are labelled as ``$L_s$''. Then, the estimates appearing in \cite{dierkCosine} regarding this partition and the maps from \eqref{eq_inversecosine} apply to our setting. In particular, we will use the following:
\begin{prop}[Properties of the partition~{\cite[Lemmas 2.3 and 3.4]{dierkCosine}}] \label{prop_fromRS08} In the setting described in \ref{dis_inversenormalized}, the following hold:
	\begin{itemize}
		\item If $z,w \in F_{(n,\ast)}$ for some $(n, \ast) \in \Z\times\{L, R\}$, then $\vert \Ima z-\Ima w\vert< 3\pi$ and moreover $\vert \Ima z-2\pi n\vert< 3\pi$.
		\item If $w \in f(\T_f) \setminus \delta$, then for each $(n, \ast) \in \Z\times\{L, R\}$ there exists $r^\star\in \C$ with $\vert r^\star \vert<1$ and such that 
		\begin{align*}
		f^{-1}_{(n, \ast)}(w)&\defeq \renewcommand{\arraystretch}{1.5}\left\{\begin{array}{@{}l@{\quad}l@{}} \log(w)-\log a + 2\pi i n + r^\star & \text{ if }\quad \ast=R;\\
		-\log(w)+ \log b + 2\pi i n + r^\star & \text{ if }\quad \ast=L. 
		\end{array}\right.\kern-\nulldelimiterspace
		\end{align*}
	\end{itemize}	
\end{prop}

\begin{defn}[External addresses]\label{def_addresses}
Let $f$ be a cosine map. An \textit{external address} is an infinite sequence $\s=F_0F_1F_2\ldots$ of the fundamental domains specified in \ref{dis_inversenormalized}. If $\s$ is such an external address, we denote
$$  J_\s= \{z\in J(f) \colon f^n(z)\in F_n \text{ for all } n\geq 0\}. $$
We let $\Addr(f)$ be the set of all $\s$ for which $J_\s$ is not empty, that we endow with the usual lexicographic cyclic order topology; see \cite[2.13]{mio_signed_addr} for details.
\end{defn}

\begin{observation}\label{rem_Jg}
If $g$ is of disjoint type, then  $J(g)= \bigcap_{k\geq 0}g^{-k}( \T_g)$ and $g^{-n}(\T_g) \subset \T_g$ for all $n\geq 0$; see \cite[Proposition 3.2]{lasse_arclike}. In particular, 
$$J(g)=\bigcup_{\s\in \Addr(g)} J_\s. $$
\end{observation}

\begin{notations}
	For each element $(n, \ast)\in \Z\times\{L, R\}$, we denote $\vert (n, \ast)\vert \defeq \vert n\vert $ and $\{(n, \ast)\}\defeq n$.
\end{notations}

\section{A model for cosine dynamics}\label{sec_model_disjoint}
\begin{discussion}[Topological space $(\M, \tau_\M)$]\label{dis_topology_model_cos}
	Consider the set
	$$\M\defeq [0,\infty)\times(\Z\times\{L, R\})^\N.$$
	Let ``$<_{_\Z}$'' be the usual linear order on integers. We define a total order in the set $\Z\times\{L, R\}$ as follows:
	\begin{align}\label{eq_order_LR}
	(n,\ast)<(m,\star)& \iff \renewcommand{\arraystretch}{1.5}\left\{\begin{array}{@{}l@{\quad}l@{}} \ast=R=\star & \text{and} \quad n<_{_\Z}m, \quad \text{ or }\\
	\ast=L=\star & \text{and} \quad m<_{_\Z}n, \quad \text{ or }\\
	\ast=L &\text{and} \quad \star=R.
	\end{array}\right.\kern-\nulldelimiterspace
	\end{align}
	Then, $<$ induces a lexicographic order ``$<_{\ell}$'' in $(\Z\times\{L, R\})^\N$. In turn, we define a cyclic order induced by $<_{\ell}$ in the usual way: for $\s,\underline{\alpha},\underline{\tau} \in (\Z\times\{L, R\})^\N$,
	$$[\s,\underline{\alpha},\underline{\tau}]_{\ell} \quad \text{if and only if} \quad \s <_{_\ell}\underline{\alpha}<_{_\ell}\underline{\tau} \quad \text{ or } \quad \underline{\alpha}<_{_\ell} \underline{\tau} <_{_\ell} \s \quad \text{ or } \quad \underline{\tau} <_{_\ell} \s <_{_\ell} \underline{\alpha}.$$
	Moreover, given two different elements $\s,\underline{\tau} \in (\Z\times\{L, R\})^\N$, we define the \textit{open interval} from $\s$ to $\underline{\tau}$, denoted by $(\s,\underline{\tau})$, as the set of all points $x\in (\Z\times\{L, R\})^\N$ such that $[\s,x,\underline{\tau}]$. The collection of all such open intervals forms a base for the cyclic order topology. We then provide the space $\M$ with the topology $\tau_\M$ defined as the product topology of $[0, \infty)$ with the usual topology, and $(\Z\times\{L, R\})^\N$ with the just described cyclic order topology.
\end{discussion}

\begin{notations}
	If for some $k\geq0$, $\s=s_0s_1s_2\ldots \in (\Z\times\{L, R\})^\N$ is such that $s_j=s_k$ for all $j>k$, then we write $\s=s_0s_1\ldots \overline{s_k}$.
\end{notations}

\begin{observation}[Correspondence between topological spaces]\label{obs_corresp_orders_cosine} Let $g$ be any disjoint type cosine map, and let $\Addr(g)$ be the set of external addresses, see Definition \ref{def_addresses}. In particular, $\Addr(g)$ is endowed with a cyclic order topology. We note that there exists a one-to-one correspondence between $(\Z\times\{L, R\})^\N$ and $\Addr(g)$ that preserves their topologies, namely, the one that converts sequences as follows
$$(m,\star)(n,\ast)\ldots	\leftrightsquigarrow F_{(m,\star)} F_{(n,\ast)}\ldots.$$	
Since the curve $\delta$ chosen in \ref{dis_inversenormalized} is a vertical straight line, the linear order in fundamental domains chosen to define the cyclic order topology in $\Addr(g)$ agrees with the linear order \eqref{eq_order_LR} that determines the topology in $(\Z\times\{L, R\})^\N$, up to the specified correspondence. Hence, from now on we omit the specification of the correspondence, and $\s$ might denote either an element $(m,\star)(n,\ast)\ldots$ of $(\Z\times\{L, R\})^\N$, or its corresponding element $F_{(m,\star)} F_{(n,\ast)}\ldots$ in $\Addr(g)$.
\end{observation}

\begin{defn}[A topological model for cosine dynamics]\label{defn_modelcosine}
	Let $(\M, \tau_\M)$ be defined as in~\ref{dis_topology_model_cos}. Define $\mathcal{F} \colon (\M,\tau_\M) \to (\M,\tau_\M)$ as 
	$$\F(t,\s)\defeq(F(t) -2\pi |s_1| , \sigma(\s)), $$
	where $\sigma$ is the shift map on one-sided infinite sequences of $(\Z\times\{L, R\})^\N$, and $F(t)\defeq e^t-1$ is the standard map that codes exponential growth. Let $T\colon \M \rightarrow [0, \infty)$ be given by $T(t, \s)\defeq t $. We set
	\begin{align*}
	J(\F) &\defeq \{ x\in \M \colon \; T(\F^n(x)) \geq 0 \text{ for all $n\geq 0$}\}, \qquad\text{and} \\
	I(\F)&\defeq \{ x \in J(\F)\colon T(\F^n(x)) \to \infty \text{ as }n\to\infty \}. 
	\end{align*}
	We say that $\s \in (\Z\times\{L, R\})^\N$ is \textit{exponentially bounded} if $(t, \s) \in J(\F)$ for some $ t>0$, and denote by $\Exad$ the set of exponentially bounded elements. We moreover let
	\begin{align*}
	t_\s&\defeq \renewcommand{\arraystretch}{1.5}\left\{\begin{array}{@{}l@{\quad}l@{}} \min\{t\geq 0 : (t, \s) \in J(\F)\} & \text{if } \s \text{ is exponentially bounded}, \\
	\infty & \text{otherwise}. 
	\end{array}\right.\kern-\nulldelimiterspace
	\end{align*}
\end{defn}	

In other words, $J(\F)$ is the set of all points that stay in the space $\M$ under $\F^n$ for all $n\geq 0$.

\begin{remark} Compare to \cite[Appendix A]{helenaSemi}, where the construction of a similar model for the map $z \mapsto \pi\sinh(z)$ is sketched.
\end{remark}

\begin{observation}[Relation between cosine and exponential models]\label{obs_relmodels} Suppose that the set $\Z^\N$ is endowed with the lexicographical order topology, and define $\M_{\exp}\defeq[0, \infty) \times \Z^\N$ with the product topology. Moreover, define the map $\F_{\exp}\colon \M_{\exp}\rightarrow \M_{\exp}$ and the set $J(\F_{\exp})$ replacing in Definition~\ref{defn_modelcosine} the space $\M$ by $\M_{\exp}$. Then,  $(\F_{\exp}, J(\F_{\exp}))$ is the model for the dynamics of exponential maps described in \cite[Section 3]{lasseTopExp} and \cite[Definition~3.1]{nada}. We note that there does not exist an order preserving bijection between $\Z^\N$ with the usual lexicographic order and $((\Z\times\{L, R\})^\N, <_\ell)$, and hence the models are not the same. This was expected, since exponential maps have a single tract contained on a right half plane, while cosine maps have two tracts, as noted in \ref{dis_basic_cosine}. However, the spaces $\M$ and $\M_{\exp} \times \{L,R\}^\N$ with the product topology are homeomorphic via the map $h \colon \M_{\exp} \times \{L,R\}^\N \rightarrow \M$ given by $h(t,\s,\ul{\omega})\defeq(t,(s_0,w_0)(s_1,w_1)(s_2,w_2)\ldots)$, where $\s=s_0s_1\ldots \in \Z^N$ and $\ul{\omega}=w_0w_1w_2\ldots\in \{L,R\}^\N$. This can be seen by recalling that a base for the product topology  of $\M_{\exp} \times \{L,R\}^\N$ is given by cylinders, and the image of each such cylinder under $h$ can be expressed as a union of intervals of $\tau_\M$, and vice-versa, preimages of intervals are unions of cylinders. In particular, $J(\F)$ is homeomorphic to $J(\F_{\exp})\times\{L,R\}^\N$, where each subspace has the topology respectively induced from $\M$ and $\M_{\exp} \times \{L,R\}^\N$.
\end{observation}

We shall use the relation specified above between the exponential and cosine models to prove properties of the latter: 

\begin{prop}[Properties of the cosine model]\label{prop_properties_cosine} The space $J(\F)$ with the induced subspace topology admits the $1$-point compactification, and the resulting space $J(\F)\cup\{\tilde{\infty}\}$ is a sequential space. Moreover, $\F \vert_{J(\F)}$ is continuous.
\end{prop}

\begin{proof}
	By Observation \ref{obs_relmodels}, $J(\F)$ is homeomorphic to $J(\F_{\exp})\times\{L,R\}^\N$. In turn, $J(\F_{\exp})$ is homeomorphic to a straight brush, which is a subset of $\R^2$ with the usual Euclidean metric, see \cite[Theorem 3.3]{nada}, and $\{L,R\}^\N$ is homeomorphic to the Cantor set. Hence, the product space $J(\F_{\exp})\times\{L,R\}^\N$ is also locally compact and Hausdorff, and so it admits the one-point compactification; see \cite[\S19 and \S29]{Munkres}. Then, as the resulting space is second, and so first, countable, it is a sequential space \cite[p. 12 Definition 9 and p.14  Proposition 7]{general_topology}.

In order to prove continuity of $\F\vert_{J(\F)}$, let us fix an arbitrary $(t,\s) \in J(\F)$ and let $V$ be an open neighbourhood of $\F(t, \s)$. Without loss of generality, we may assume that $V=((t_1, t_2)\times \mathcal{I}) \cap J(\F)$ for some open interval $\mathcal{I} \in (\Z\times \{L, R\})^\N$ and $t_1, t_2\in \R^+$ so that $t_1\leq T(\F(t,\s))< t_2$. Suppose that $\s=s_0s_1\ldots$ and denote $\tilde{\mathcal{I}}\defeq \{s_0\ultau : \ultau \in \mathcal{I}\}$. In particular, $\s \in \tilde{\mathcal{I}}$, and since by definition of $\F$, $t=\log(T(\F(t, \s))+1+2\pi\{s_1\})$ and the function $\log$ is increasing, 
	$$U\defeq(\log(t_1+1+2\pi \{s_1\} ), \log(t_2+1+2\pi \{s_1\}))\times \tilde{\mathcal{I}})\cap J(\F)$$
	is an open neighbourhood of $(t,\s)$ such that $\F(U)\subset V.$ 
\end{proof} 

In order to prove Theorem \ref{thm_conjcosine_intro}, our main task will be to construct for each disjoint type cosine map $g$, a continuous map $\Phi: J(\F) \rightarrow J(g)$ that conjugates the dynamics of $\F$ to those of $g\vert_{ J(g)}$. Then, the result for any cosine map $f$ will follow using Rempe's conjugacy between $f$ and $g$ ``near infinity''; \cite{lasseRigidity}. In the disjoint type case, the map $\Phi$ will send each point $(t,\ul{s})\in J(\F)$ to a point $z\in J(g)$ such that $z\in J_\s$ and $ \vert \Rea z\vert \approx~t$, see Observation \ref{obs_corresp_orders_cosine}. We will obtain the map $\Phi$ as the limit of a series of \textit{approximations} $\{ \Phi_n\}_{n\in \N}$; compare \cite{lasseTopExp,lasseRigidity,helenaSemi,mio_splitting} for similar arguments. The first approximation should be a projection from the space $J(\F)$ to the dynamical plane of~$g$.

\begin{defn}[Projection function] For each $A\geq 0$, we define a \textit{projection function} $\mathcal{C}_A: J(\F) \to \C$ as 
	\begin{align*}
	\mathcal{C}_A(t, \s)&\defeq \renewcommand{\arraystretch}{1.5}\left\{\begin{array}{@{}l@{\quad}l@{}} t+A + 2\pi \{s_0\}i & \text{if} \quad s_0=(n,R) \text{ for some }n\in \Z, \\
	-t-A+ 2\pi \{s_0\}i & \text{otherwise}, 
	\end{array}\right.\kern-\nulldelimiterspace
	\end{align*}
	where $\s=s_0s_1\ldots$, and if $s_0=(n, \ast)$, then $\{s_0\}=\{(n, \ast)\}=n$.
\end{defn}

\begin{observation}[The projection of $J(\F)$ lies in fundamental domains]\label{obs_CA}
Suppose that $g$ is a disjoint type cosine function for which fundamental domains have been defined following \ref{dis_inversenormalized}. If $A$ is the constant from Observation \ref{obs_constantA}, then $\mathcal{C}_A(J(\F))$ is totally contained in the union of fundamental domains. More specifically, for each $(t, \s)\in J(\F)$, if $\s=s_0s_1\ldots$, then $\mathcal{C}_A(t, \s) \subset F_{s_0}$; see also Observation \ref{obs_corresp_orders_cosine}.
\end{observation}

\begin{remark} The reason why instead of \textit{projecting} under $\mathcal{C}_A$ each point $(t, \s)\in J(\F)$ to a point of real part $\pm t$, but rather $\pm t \pm A$ for some constant $A$, is to ensure that for a fixed function $g$, the image of each $(t,\s)\in J(\F)$ under a projection map lies in a fundamental domain of $g$, on which, by Proposition \ref{prop_fromRS08}, $g$ \textit{expands} the Euclidean metric. Note that $\mathcal{C}_A(J(\F)) \nsubseteq J(g)$. Nonetheless, since $\Phi$ will be obtained as the limit of a composition of functions consisting of inverse branches of $g$ whose images lie in $\T_g$, by Observation \ref{rem_Jg}, its codomain will be $J(g)$.
\end{remark}

Recall that cosine maps behave like the exponential map for points with modulus large enough and sufficiently far from the imaginary axis. In particular, all such points are contained in fundamental domains. An essential characteristic of our model for cosine dynamics is that, as occurs for the exponential model, for each $(t,\s)\in J(\F)$, $\vert \mathcal{C}_A(\F(\ul{s},t))\vert$ is 
roughly the exponential of its real part. More precisely:
\begin{prop}[Model acts similar to the exponential]\label{sizeC} If $(t,\s) \in J(\F)$ and $A>0$, 
	\begin{equation}\label{eq_sizeC}
	\frac{F(t)+A}{\sqrt{2}} \leq \vert \mathcal{C}_A(\F(t,\s))\vert \leq F(t)+A.
	\end{equation}
\end{prop}
\begin{proof} Suppose that $\s=s_0s_1\ldots $ and let $b\defeq 2\pi \{s_1\} $. Then,
	\begin{equation} \label{eq_squareroot}
	\begin{split}
	\vert \mathcal{C}_A(\F(t,\s))\vert &= \vert \pm(F(t)-b+A) + ib \vert=\sqrt{(F(t)+A-b)^{2}+ b^2} \\ &=\sqrt{(F(t)+A)^{2}-2(F(t)+A)b +2b^{2}}.
	\end{split}
	\end{equation}
	The second inequality in \eqref{eq_sizeC} follows from the assumption $T(\F(t,\s))\geq 0$, that is, $F(t)-b \geq ~0$, because by \eqref{eq_squareroot},
	$$\vert \mathcal{C}_A(\F(t,\s))\vert\leq \sqrt{(F(t)+A)^{2}} \iff -2(F(t)+A)b +2b^{2}\leq 0 \iff b \leq F(t)+A,$$
	where we have used that $A,b,F(t)\geq 0$. For the first inequality in \eqref{eq_sizeC}, we have
	\begin{equation*} 
	\begin{split}
	\sqrt{(F(t)+A)^{2}} \leq \sqrt{2} \vert \mathcal{C}_A(\F(t,\s))\vert & \iff (F(t)+A)^{2}-4(F(t)+A)b +4b^{2} \geq 0 \\ &\iff(F(t)+A-2b)^{2}\geq 0.\qedhere
	\end{split}
	\end{equation*}
\end{proof}

We describe the underlying idea in the construction of the map $\Phi$ that conjugates $\F$ to any disjoint type map $g\vert_{J(g)}$. For each $n\geq 0$, a function $\Phi_n\colon J(\F) \rightarrow \C$ will be defined the following way: we iterate each point $x=(t, \s)\in J(\F)$, with $\s=s_0s_1\ldots$, under the model function $\F$ a number $n$ of times. In particular, $\F^n(t, \s)=(t', \sigma^n(\s))$ for some $t'>0$. Next, we \textit{move} to the dynamical plane of $g$ using the function $\mathcal{C}_A$ for some constant $A$ big enough such that $(\mathcal{C}_A\circ\F^n)(t, \s)\in F_{s_n}$. Then, we use the composition of $n$ inverse branches of $g$ specified in \eqref{eq_inversecosine} to obtain a point in $F_{s_0}$, which will define $\Phi_n(x)$; see Figure \ref{fig:semi_cosine}. Finally, we use (Euclidean) expansion of $g$ on its tracts to show that $\{\Phi_n\}_{n\geq 0}$ is a uniformly convergent sequence. We now formalize these ideas:

\begin{defn}[Functions $\Phi_n$]\label{def_Phin} Let $g$ be a disjoint-type cosine map, and let $A$ be a constant provided by Observation \ref{obs_constantA}. Then, for each $n\geq 0$ we define the function
	$\Phi_{n}: J(\F)\rightarrow \C$ as
	$$\Phi_0(x)\defeq \mathcal{C}_A(x)\qquad \text{ and }\qquad \Phi_{n+1}(x)\defeq g^{-1}_{s_0}(\Phi_{n}(\F(x)), $$	
	for $x=(t,\s)$ and $\s=s_0s_1\ldots$.
\end{defn}

The function $\Phi_0$ is clearly well-defined. In order to see that for all $n\geq 1$ the function $\Phi_n$ is also well-defined, fix $x=(t, \s)\in J(\F)$ and suppose that $\s=s_0s_1\ldots$. Then, expanding definitions 
\begin{equation}\label{eq_expdef_cos}
\Phi_{n}(x)= \left(g^{-1}_{s_0}\circ g^{-1}_{s_1} \circ \cdots \circ g^{-1}_{s_{n-1}} \circ \mathcal{C}_A \circ \F^n\right)(x). 
\end{equation}
By Observations \ref{obs_CA} and \ref{rem_Jg}, the composition of the inverse branches $\{g^{-1}_{s_i}\}_{i< n}$ is well-defined on $\mathcal{C}_A(\F^n(x))\in F_{s_n}$. Moreover, by construction, for all $n\geq 0$,
\begin{equation} \label{commute_cos}
\Phi_{n}\circ \F= g \circ \Phi_{n+1}.
\end{equation}

\begin{figure}[h]
	\centering
\begingroup%
\makeatletter%
\providecommand\color[2][]{%
	\errmessage{(Inkscape) Color is used for the text in Inkscape, but the package 'color.sty' is not loaded}%
	\renewcommand\color[2][]{}%
}%
\providecommand\transparent[1]{%
	\errmessage{(Inkscape) Transparency is used (non-zero) for the text in Inkscape, but the package 'transparent.sty' is not loaded}%
	\renewcommand\transparent[1]{}%
}%
\providecommand\rotatebox[2]{#2}%
\newcommand*\fsize{\dimexpr\f@size pt\relax}%
\newcommand*\lineheight[1]{\fontsize{\fsize}{#1\fsize}\selectfont}%
\ifx\svgwidth\undefined%
\setlength{\unitlength}{396.8503937bp}%
\ifx\svgscale\undefined%
\relax%
\else%
\setlength{\unitlength}{\unitlength * \real{\svgscale}}%
\fi%
\else%
\setlength{\unitlength}{\svgwidth}%
\fi%
\global\let\svgwidth\undefined%
\global\let\svgscale\undefined%
\makeatother%
\begin{picture}(1,0.57142857)%
\lineheight{1}%
\setlength\tabcolsep{0pt}%
\put(0,0){\includegraphics[width=\unitlength,page=1]{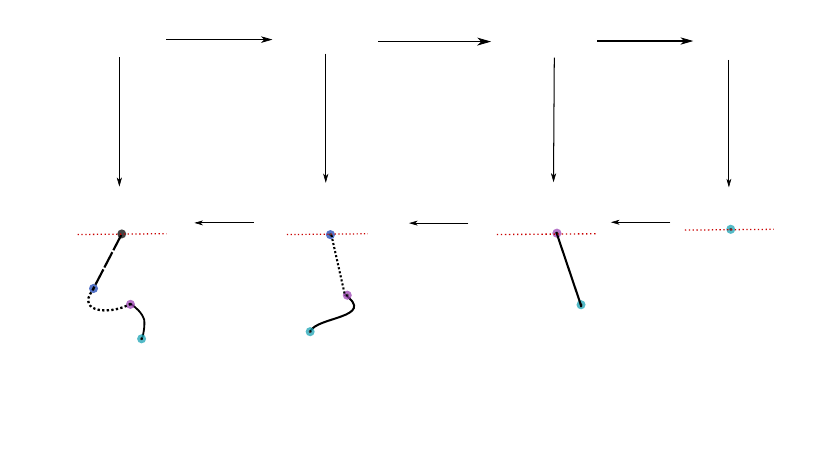}}%
\put(0.10665545,0.43757785){\color[rgb]{0,0,0}\makebox(0,0)[lt]{\lineheight{1.25}\smash{\begin{tabular}[t]{l}$\mathcal{C}_A$\\\end{tabular}}}}%
\put(0.13767613,0.51873606){\color[rgb]{0,0,0}\makebox(0,0)[lt]{\lineheight{1.25}\smash{\begin{tabular}[t]{l}$x$\end{tabular}}}}%
\put(0.264605,0.53325616){\color[rgb]{0,0,0}\makebox(0,0)[lt]{\lineheight{1.25}\smash{\begin{tabular}[t]{l}$\mathcal{F}$\end{tabular}}}}%
\put(0.3476481,0.51862544){\color[rgb]{0,0,0}\makebox(0,0)[lt]{\lineheight{1.25}\smash{\begin{tabular}[t]{l}$\mathcal{F}(x)$\end{tabular}}}}%
\put(0.4997473,0.53119302){\color[rgb]{0,0,0}\makebox(0,0)[lt]{\lineheight{1.25}\smash{\begin{tabular}[t]{l}$\mathcal{F}$\end{tabular}}}}%
\put(0.60570297,0.51856733){\color[rgb]{0,0,0}\makebox(0,0)[lt]{\lineheight{1.25}\smash{\begin{tabular}[t]{l}$\mathcal{F}^2(x)$\end{tabular}}}}%
\put(0.63125501,0.44135269){\color[rgb]{0,0,0}\makebox(0,0)[lt]{\lineheight{1.25}\smash{\begin{tabular}[t]{l}$\mathcal{C}_A$\\\end{tabular}}}}%
\put(0.10473663,0.30289369){\color[rgb]{0,0,0}\makebox(0,0)[lt]{\lineheight{1.25}\smash{\begin{tabular}[t]{l}$\Phi_0(x)$\\\end{tabular}}}}%
\put(0.01464967,0.21501405){\color[rgb]{0,0,0}\makebox(0,0)[lt]{\lineheight{1.25}\smash{\begin{tabular}[t]{l}$\Phi_1(x)$\\\end{tabular}}}}%
\put(0.17390663,0.19747265){\color[rgb]{0,0,0}\makebox(0,0)[lt]{\lineheight{1.25}\smash{\begin{tabular}[t]{l}$\Phi_2(x)$\\\end{tabular}}}}%
\put(0.15468234,0.13427909){\color[rgb]{0,0,0}\makebox(0,0)[lt]{\lineheight{1.25}\smash{\begin{tabular}[t]{l}$\Phi_3(x)$\\\end{tabular}}}}%
\put(0.33779994,0.3017161){\color[rgb]{0,0,0}\makebox(0,0)[lt]{\lineheight{1.25}\smash{\begin{tabular}[t]{l}$\Phi_0(\mathcal{F}(x))$\\\end{tabular}}}}%
\put(0.43926888,0.20232222){\color[rgb]{0,0,0}\makebox(0,0)[lt]{\lineheight{1.25}\smash{\begin{tabular}[t]{l}$\Phi_1(\mathcal{F}(x))$\\\end{tabular}}}}%
\put(0.37657855,0.14482809){\color[rgb]{0,0,0}\makebox(0,0)[lt]{\lineheight{1.25}\smash{\begin{tabular}[t]{l}$\Phi_2(\mathcal{F}(x))$\\\end{tabular}}}}%
\put(0.76495449,0.53612492){\color[rgb]{0,0,0}\makebox(0,0)[lt]{\lineheight{1.25}\smash{\begin{tabular}[t]{l}$\mathcal{F}$\end{tabular}}}}%
\put(0.84989134,0.51512522){\color[rgb]{0,0,0}\makebox(0,0)[lt]{\lineheight{1.25}\smash{\begin{tabular}[t]{l}$\mathcal{F}^3(x)$\end{tabular}}}}%
\put(0.35387366,0.43806512){\color[rgb]{0,0,0}\makebox(0,0)[lt]{\lineheight{1.25}\smash{\begin{tabular}[t]{l}$\mathcal{C}_A$\\\end{tabular}}}}%
\put(0.83402176,0.44268992){\color[rgb]{0,0,0}\makebox(0,0)[lt]{\lineheight{1.25}\smash{\begin{tabular}[t]{l}$\mathcal{C}_A$\\\end{tabular}}}}%
\put(0.60657365,0.30406562){\color[rgb]{0,0,0}\makebox(0,0)[lt]{\lineheight{1.25}\smash{\begin{tabular}[t]{l}$\Phi_0(\mathcal{F}^2(x))$\\\end{tabular}}}}%
\put(0.67791612,0.16450898){\color[rgb]{0,0,0}\makebox(0,0)[lt]{\lineheight{1.25}\smash{\begin{tabular}[t]{l}$\Phi_1(\mathcal{F}^2(x))$\\\end{tabular}}}}%
\put(0.82758784,0.30889696){\color[rgb]{0,0,0}\makebox(0,0)[lt]{\lineheight{1.25}\smash{\begin{tabular}[t]{l}$\Phi_0(\mathcal{F}^3(x))$\\\end{tabular}}}}%
\put(0.14974553,0.03177616){\color[rgb]{0.4,0.4,0.4}\makebox(0,0)[lt]{\lineheight{1.25}\smash{\begin{tabular}[t]{l}\fontsize{9pt}{1em}$F_{s_0}$\end{tabular}}}}%
\put(0.8005534,-1.66183908){\color[rgb]{0,0,0}\makebox(0,0)[lt]{\begin{minipage}{0.01730374\unitlength}\raggedright \end{minipage}}}%
\put(0.37874312,0.03035965){\color[rgb]{0.4,0.4,0.4}\makebox(0,0)[lt]{\lineheight{1.25}\smash{\begin{tabular}[t]{l}\fontsize{9pt}{1em}$F_{s_1}$\end{tabular}}}}%
\put(0.62498093,0.03042058){\color[rgb]{0.4,0.4,0.4}\makebox(0,0)[lt]{\lineheight{1.25}\smash{\begin{tabular}[t]{l}\fontsize{9pt}{1em}$F_{s_2}$\end{tabular}}}}%
\put(0.88428426,0.02664076){\color[rgb]{0.4,0.4,0.4}\makebox(0,0)[lt]{\lineheight{1.25}\smash{\begin{tabular}[t]{l}\fontsize{9pt}{1em}$F_{s_3}$\end{tabular}}}}%
\put(0.91934041,0.26712268){\color[rgb]{0.83137255,0,0}\makebox(0,0)[lt]{\lineheight{1.25}\smash{\begin{tabular}[t]{l}\fontsize{7pt}{1em}$2\pi i \{ s_3 \}$\end{tabular}}}}%
\put(0.71751317,0.26674366){\color[rgb]{0.83137255,0,0}\makebox(0,0)[lt]{\lineheight{1.25}\smash{\begin{tabular}[t]{l}\fontsize{7pt}{1em}$2\pi i \{ s_2 \}$\end{tabular}}}}%
\put(0.18357416,0.26982686){\color[rgb]{0.83137255,0,0}\makebox(0,0)[lt]{\lineheight{1.25}\smash{\begin{tabular}[t]{l}\fontsize{7pt}{1em}$2\pi i \{s_0 \}$\end{tabular}}}}%
\put(0.44594048,0.26898738){\color[rgb]{0.83137255,0,0}\makebox(0,0)[lt]{\lineheight{1.25}\smash{\begin{tabular}[t]{l}\fontsize{7pt}{1em}$2\pi i \{ s_1 \}$\end{tabular}}}}%
\put(0,0){\includegraphics[width=\unitlength,page=2]{Semiconjugacy_cosine.pdf}}%
\put(0.74904109,0.31880455){\color[rgb]{0,0,0}\makebox(0,0)[lt]{\lineheight{1.25}\smash{\begin{tabular}[t]{l}\fontsize{9pt}{1em} $g^{-1}_{s_2}$\end{tabular}}}}%
\put(0.50360887,0.32036866){\color[rgb]{0,0,0}\makebox(0,0)[lt]{\lineheight{1.25}\smash{\begin{tabular}[t]{l}\fontsize{9pt}{1em} $g^{-1}_{s_1}$\end{tabular}}}}%
\put(0.24276889,0.32250786){\color[rgb]{0,0,0}\makebox(0,0)[lt]{\lineheight{1.25}\smash{\begin{tabular}[t]{l}\fontsize{9pt}{1em} $g^{-1}_{s_0}$\end{tabular}}}}%
\end{picture}%
\endgroup%

	\caption{A schematic of the functions and curves involved in the definition of the functions $\{\Phi_n\}_{n\in \N}$.}
	\label{fig:semi_cosine}
\end{figure}

\begin{prop}[Continuity of the functions $\Phi_n$] \label{prop_cont_Phin} For each $n\geq 0$, $\Phi_n~\colon~J(\F) \rightarrow \C$ is continuous.
\end{prop}
\begin{proof} Let us fix an arbitrary $(t,\s)\in J(\F)$ with $\s=s_0s_1\ldots$, as well as some $\epsilon>0$. To see that $\Phi_0\equiv\mathcal{C}_A$ is continuous, let $\mathcal{I} \subset (\Z\times\{L, R\})^\N$ be any open interval containing $\s$ and such that if $\ultau=\tau_0\tau_1\ldots \in \mathcal{I}$, then $s_0=\tau_0$. Then, $U\defeq \left((t-\epsilon, t+\epsilon) \times \mathcal{I}\right) \cap J(\F)$ is an open neighbourhood of $(t, \s)$ such that
	$$\mathcal{C}_A(U)\subset (\pm t-\epsilon, \pm t+\epsilon) \pm A + 2\pi i \{s_0\} \subset \D_\epsilon(\pm t\pm A + 2\pi i \{s_0\})= \D_\epsilon(\mathcal{C}_A(t,\s)),$$	
	where $\pm$	equals ``$+$'' or ``$-$'' depending on whether $s_0=(n,R)$ or $s_0=(n,L)$ for some $n\in \Z$. Hence, we have shown continuity of $\Phi_0$. For each $n\geq 1$, let
	$$L_n\defeq g_{s_0}^{-1} \circ g_{s_1}^{-1} \circ \cdots \circ g^{-1}_{s_{n-1}}\circ \Phi_0 \circ \F^n$$
	and note that for any subset $U\subset J(\F)$ such that $\Phi_0(\F^n(U))\subset F_{s_n}$, by Proposition \ref{prop_properties_cosine} and the definition of the maps $\{g^{-1}_{s_i}\}_{i< n}$, $L_n\vert_{U}$ is a continuous function, as it is a composition of continuous functions. By \eqref{eq_expdef_cos}, $L_n(t,\s)=\Phi_n(t,\s)$. Hence, in order to prove continuity of $\Phi_n$ at $(t,\s)$, since by Observation \ref{obs_CA} $\Phi_0(\F^n(t,\s))\subset F_{s_n}$, it suffices to find a neighbourhood $V \ni (t,\s)$ such that $\Phi_{n}\vert_{V} \equiv L_{n}\vert_{V}$. Let $J_n \subset (\Z\times\{L, R\})^\N$ be any open interval containing $\s$ and such that if $\ultau=\tau_0\tau_1\ldots \in J_n$, then $s_i=\tau_i$ for all $0\leq i \leq n$, and choose $t_1,t_2\in \R^+$ so that $t_1\leq t \leq t_2$. Then, $V\defeq ((t_1, t_2)\times J_n) \cap J(\F)$ satisfies the properties required and continuity of $\Phi_n$ follows.
\end{proof}

We are now ready to  prove Theorem \ref{thm_conjcosine_intro}. We will do so by first showing that for any given disjoint-type cosine map $g$, the functions $\{\Phi_n\}_{n\in \N}$ converge to a continuous function $\Phi\colon J(\F)\to J(g)$ satisfying the properties listed on the second part of the statement. Then, for the first part, given any cosine map $f$, we will use \cite[Theorem 4.6]{mio_newCB} to relate its dynamics to those of a disjoint-type cosine map, and then apply the second part. We note that  \cite[Theorem 4.6]{mio_newCB} relies heavily on Rempe's analogs of Bottcher's map for transcendental maps, and more specifically on \cite[Theorem 3.2]{lasseRigidity}. 
\begin{proof}[Proof  of Theorem \ref{thm_conjcosine_intro}]
	We start by showing the second part of the statement. Let $g$ be a disjoint type cosine map and let $\T_g$ be a pair of expansion tracts for $g$. Let $\{\Phi_{n}\}_{n\geq 0}$ be the sequence of functions from Definition \ref{def_Phin}, and suppose that $g(z)=ae^z+ be^{-z}$ for some $a,b\in \C^\ast$. If $M\defeq \max\{\vert a \vert, \vert b \vert\}$, then by Propositions \ref{prop_fromRS08} and \ref{sizeC}, for each $x=(t, \s) \in J(\F)$ with $\s=s_0s_1s_2\ldots$,
	\begin{equation*}
	\vert \Rea( \Phi_1(x)) \vert =\vert \Rea( (g_{s_0}^{-1} \circ \mathcal{C}_A \circ\F)(x))\vert \leq \ln\vert \mathcal{C}_A(\F(x)) \vert +\vert \ln(M) \vert +1 \leq  t+A +\vert\ln(M) \vert+2. 
	\end{equation*}
	Similarly, $\vert \Rea( \Phi_1(x)) \vert \geq t-\ln(\sqrt{2})-\vert\ln(M) \vert-2$. By the definition of $\Phi_1$ and Observation ~\ref{obs_constantA}, both $\Phi_0(x)$ and $\Phi_1(x)$ lie in the same fundamental domain $F_{s_0}$. Thus, either both points have positive real part, of both have negative real part. By this and using that $\vert \Rea (\Phi_0(x))\vert = t+A$, we have
	\begin{equation}\label{eq_boundreal}
	\vert \Rea (\Phi_0(x)) - \Rea (\Phi_1(x))\vert\leq A+\ln(\sqrt{2}) +\vert \ln(M) \vert+2.
	\end{equation}
	Moreover, by \eqref{eq_boundreal} and Proposition \ref{prop_fromRS08},
	\begin{equation}\label{eq_mucosine}
	\vert \Phi_0(x) -\Phi_1(x)\vert \leq A +\ln(\sqrt{2}) + \vert \ln(M)\vert +2 + 3\pi \eqdef\mu,
	\end{equation}
	where we note that the constant $\mu$ does not depend on the point $x$. In particular, $\Phi_0(x)$ and $\Phi_1(x)$ lie in the same tract, and hence the straight segment joining these two points is totally contained in a connected component of $\{z: \vert \Rea z\vert >\mathcal{K}(g)\}$, which is a convex set. Moreover, by \eqref{eq_expdef_cos}, if $\mathsf{g}^{-1}_{\s,n}\defeq g^{-1}_{s_0}\circ g^{-1}_{s_1} \circ \cdots \circ g^{-1}_{s_{n-1}}$ for some $n\geq 1$, then $$\Phi_n(x)=(\mathsf{g}^{-1}_{\s,n}\circ \Phi_0 \circ \F^n)(x) \quad \text{ and } \quad \Phi_{n+1}(x)=(\mathsf{g}^{-1}_{\s,n}\circ\Phi_1 \circ\F^n)(x).$$
	Note that if $\gamma$ is the straight segment connecting $\Phi_0(\F^n(x))$ and $\Phi_1(\F^n(x))$, then since the map $\mathsf{g}^{-1}_{\s,n}$ is a bijection to its image as it is a composition of bijections, $\mathsf{g}^{-1}_{\s,n}
	(\gamma)$ is a curve with endpoints $\Phi_n(x)$ and $\Phi_{n+1}(x)$. Thus, using \eqref{eq_mucosine} and \eqref{eq_expansion},
	\begin{equation}\label{eq_contractioncos}
	\vert \Phi_{n+1}(x)-\Phi_n(x)\vert \leq \frac{\vert \Phi_0(\F^n(x)) -\Phi_1(\F^n(x))\vert} {2^n} \leq \frac {\mu} {2^n}.
	\end{equation}
	Hence, $\{\Phi_n\}_{n\geq 0}$ is a uniformly Cauchy sequence of continuous functions, and so they converge uniformly to a continuous limit
	function $\Phi :J(\F) \rightarrow \C$, that by \eqref{commute_cos} satisfies 
	\begin{equation}\label{eq_commute_cos}
	\Phi\circ \F=g \circ \Phi.
	\end{equation}
	Note that for each $x \in J(\F)$, $\Phi(x)$ is the limit of the backward orbit of a point in $\T_g$, see \eqref{eq_expdef_cos}. Hence, by Observation \ref{rem_Jg}, $\Phi(x)\in J(g)$ and thus, $\Phi(J(\F))\subset J(g)$. Moreover, since $\mathcal{C}_A \equiv \Phi_0$, for each $x\in J(\F)$, 
	\begin{equation}\label{eq_PhiminusC}
	\vert \Phi(x)-\mathcal{C}_A(x)\vert \leq \sum_{n=0} ^{\infty} \vert\Phi_{n+1}(x)-\Phi_n(x)\vert \leq\sum_{j=0} ^{\infty}\frac{\mu}{2^{n}} 
	=2\mu.
	\end{equation}
	This means for any sequence $\lbrace x_n \rbrace_{n\in \N} \subset J(\F)$ that $\Phi(x_n)\rightarrow\infty$ if and only if $\mathcal{C}_A(x_n)\rightarrow\infty$ as $n\rightarrow \infty$. By this, the definition of $I(\F)$ and Proposition \ref{sizeC},
	\begin{equation}\label{corrI}
	\begin{split}
	x \in I(\F)& \Leftrightarrow \lim_{n \to \infty} T(\F^n(x))=\infty \Leftrightarrow \lim_{n \to \infty}\vert \mathcal{C}_A(\F^n(x)) \vert
	=\infty \\ & \Leftrightarrow \lim_{n \to \infty} \Phi(\F^n(x))= \lim_{n \to \infty} g^{n}(\Phi(x))=\infty 
	\Leftrightarrow \Phi(x) \in I(g).
	\end{split}
	\end{equation}
	Equivalently, $\Phi(I(\F)) \subseteq I(g)$ and $\Phi(J(\F) \setminus I(\F)) \subseteq J(g)\setminus I(g)$. Consequently, surjectivity of $\Phi$ would imply $\Phi(I(\F)) = I(g).$
	
	\begin{Claim}
		The function $\Phi: J(\F) \rightarrow J(g)$ is surjective. 
	\end{Claim} 
	\begin{subproof}
		Fix an arbitrary $z\in J(g)$. Then, $z\in J_\s$ for some $\s=s_0s_1s_2\ldots \in \Addr(g)$, see Observation \ref{rem_Jg}. Note that by definition, the function $\F$ is injective on its first coordinate. That is, for each fixed $\s\in (\Z\times \{L,R\})^\N$, $\F_{\s}\defeq \F(\cdot, \s)\colon \R^+\rightarrow \C$ given by $t\mapsto\F(t, \s)$ is injective. Hence, we can consider the sequence of real positive numbers $\{t_k\}_{k\geq 0}$ uniquely determined by the equations
		\begin{equation*}
		\F^{k}(t_k,\s)=(\vert \Rea( g^{k}(z))\vert, \sigma^{k}(\s)).
		\end{equation*}
		In particular, $T(\F^{k}(t_k,\s))=\vert\Rea(g^{k}(z))\vert>0$, and hence one can see using a recursive argument that for all $0\leq j\leq k$, 
		\begin{equation}\label{eq_Tmorethan0}
		T(\F^{j}(t_k,\s))=\log \left( T(\F^{j+1}(t_k,\s))+2\pi \{ s_{j+1}\} +1 \right)>0,
		\end{equation}
		and so $\F^j(t_k, \s)$ is indeed well-defined for all $j\leq k$. By definition of the map $\mathcal{C}_A$, it holds that $\Rea( \mathcal{C}_A(\F^{k}(t_k,\s)))=\pm\Rea(g^{k}(z)) \pm A$, where ``$\pm$'' equals ``$+$'' or ``$-$'' depending on whether $s_k$ belongs to $\Z\times \{R\}$ or $\Z\times \{L\}$. Moreover, by Observation \ref{obs_CA}, both $g^{k}(z),\mathcal{C}_A(\F^{k}(t_k,\s)) \in F_{s_k}$, and hence by Proposition \ref{prop_fromRS08},
		\begin{equation}\label{eq_bound3piA}
		\vert \mathcal{C}_A(\F^{k}(t_k,\s)) -g^{k}(z) \vert< 3\pi + A.
		\end{equation}
		Note that for all $j\leq k$, since the second coordinate of $\F^{j}(t_k,\s)$ equals $\sigma^j(\s)$, we have that $\Phi_{k-j}(\F^{j}(t_k,\s))=(g^{-1}_{s_{j}}\circ \cdots \circ g^{-1}_{s_{k-1}}\circ \mathcal{C}_A \circ\F^{k})(t_k,\s)$ and $g^{j}(z)=(g^{-1}_{s_{j}}\circ \cdots \circ g^{-1}_{s_{k-1}}\circ g^k)(z)$. Hence, using \eqref{eq_expansion},  \eqref{eq_PhiminusC} and \eqref{eq_bound3piA}, by the same contraction argument as when showing \eqref{eq_contractioncos}, for any $j\leq k$,
		\begin{equation*} 
		\begin{split}
		\vert \mathcal{C}_A(\F^{j}(t_k,\s)) -g^{j}(z) \vert & \! \leq \! \vert \mathcal{C}_A(\F^{j}(t_k,\s)) \!-\!\Phi_{k-j}(\F^{j}(t_k,\s)) \vert\! + \!\vert \Phi_{k-j}(\F^{j}(t_k,\s))\!-\!g^j(z)\vert \\ &<2\mu + \frac{3\pi + A}{2^{k-j}}< 2(\mu +3\pi + A)\eqdef\eta.
		\end{split}
		\end{equation*}
		We note that the constant $\eta$ does not depend on $k$. In particular, by taking $j=0$ we see that $t_k$ is uniformly bounded from above by a constant independent of $k$, and thus $t_k\nrightarrow \infty$ as $k \rightarrow \infty$. This means that there exists at least one finite limit point for the sequence $\{t_k\}_{k\geq 0}$, say $t\geq 0$, that by \eqref{eq_Tmorethan0} satisfies $(t,\s)\in J(\F)$. Since by \eqref{eq_commute_cos}, for each $j\geq 0$, we have $g^j(\Phi(t,\s))=\Phi (\F^j(t,\s))$,
		\begin{equation*} 
		\begin{split}
		\vert g^j(\Phi(t,\s)) -g^{j}(z) \vert \! \leq \!\vert \Phi(\F^j(t,\s))\!- \!\mathcal{C}_A(\F^{j}(t,\s)) \vert \!+\! \vert \mathcal{C}_A(\F^{j}(t,\s))-g^j(z)\vert \!< \!2\mu + \eta,
		\end{split}
		\end{equation*}
		and this upper bound does not depend on $j$. Since $g^j(\Phi(t,\s))$ and $g^{j}(z)$ belong to the same fundamental domain $F_{s_j}$ for each $j\geq 0$, we can once more use the same contraction argument to conclude that the points $\Phi(t,\s)$ and $z$ are equal. 
	\end{subproof}
	
	To prove injectivity of $\Phi$, note that if $(t,\s), (t', \s) \in J(\F)$ for some $t\neq t'$, the orbits of $(t,\s)$ and $(t', \s)$ under $\F$ will eventually be far apart by definition of $\F$. Then, by \eqref{eq_commute_cos} and \eqref{eq_PhiminusC}, so will be the orbits under $g$ of $\Phi(t, \s)$ and $\Phi(t', \s)$, and injectivity follows.
	
	Since the compactification $J(\F)\cup \{\tilde{\infty}\}$ is by Proposition \ref{prop_properties_cosine} a sequential space, and so is $\widehat{\C}\defeq \C \cup \{\infty\} $, the notions of continuity and sequential continuity for functions between these spaces are equivalent. Thus, using \eqref{corrI}, we can extend $\Phi$ to a continuous map $\tilde{\Phi}: J(\F) \cup \lbrace \tilde{\infty} \rbrace \rightarrow J(g) \cup \lbrace \infty \rbrace$ by defining $\tilde{\Phi}(\tilde{\infty})\defeq\infty$. Then, $\tilde{\Phi}^{-1}$ is continuous as it is the inverse of a continuous bijective map on a compact space, and consequently, by respectively removing $\tilde{\infty}$ and $\infty$ from the domain and codomain of $\tilde{\Phi}^{-1}$, it follows that $\Phi^{-1}$ is also continuous.
	
	We have shown that for any cosine map $g$ of disjoint type, $\Phi \colon J(\F) \rightarrow J(g)$ is a homeomorphism and $\Phi(I(F))=I(g)$. 
	
	Suppose now that $f$ is any cosine map. If follows from \cite[Theorem 4.6]{mio_newCB} that there is a disjoint type cosine map $g$ and a continuous function $\theta\colon J(g)\to J(f)$ with the following properties:
	\begin{enumerate}
	\item  \label{enum1}The restriction $\theta\colon  \theta^{-1}(J_K(f)) \to J_K(f)$ is a homeomorphism for some $K>0$;
	\item \label{enum2} $\theta \circ g=f\circ \theta$ in $\theta^{-1}(J_K(f))$; 
	\item \label{enum3} $\theta(I(g))\subset I(f)$.
	\end{enumerate}
	
	More specifically, the map $g$ in \cite[Theorem 4.6]{mio_newCB} is of the form $\lambda f$  for some $\lambda \in \C\setminus \{0\}$, and so $g$ is a cosine map. Then, \cite[Theorem 4.6]{mio_newCB} states that $J(g)$ contains a so-called \textit{strongly absorbing Cantor bouquet} $X$,  \cite[Definition~1.5]{mio_newCB}, such that $\theta\vert_X$ is a homeomorphism to its image, \cite[Theorem 4.6(d)]{mio_newCB}, and so that $\theta(X)\supset J_K(f)$ for some $K>0$, \cite[line 18 in the proof of Theorem 4.6]{mio_newCB}. Thus, \eqref{enum1} follows. Item \eqref{enum2} is a consequence of this together with \cite[Theorem 4.6(b)]{mio_newCB}, and \eqref{enum3} is part of \cite[Theorem 4.6(g)]{mio_newCB}. Then, if $\tilde{\Phi}\colon J(\F) \rightarrow J(g)$  is the homeomorphism provided by the second part of the statement, the composition  $\Phi \defeq \theta\circ \tilde{\Phi} \colon J(\F)\to J(f)$ satisfies the desired properties.
\end{proof}

\begin{observation}\label{obs_model_bouquet}
Recall from Observation \ref{rem_Jg} that  if $g$ is of disjoint type, then $J(g)=\bigcup_{\s\in \Addr(g)} J_\s$. It follows from the proof of Theorem \ref{thm_conjcosine_intro} that in this case, for each $\s\in \Addr(g)$, $\Phi\colon \{(t, \s) \colon  t \geq t_\s\} \to J_\s$ is a bijection, see also Observation \ref{obs_corresp_orders_cosine}. This implies that $\Phi$ acts as an order-preserving map from $\Exad$ to $\Addr(g)$. Note also that $J(\F)$ is a straight brush, see the proof of Proposition~\ref{prop_properties_cosine}, and in particular a collection of disjoint unbounded curves. Hence, our result provides another proof of the fact that $J(g)$ is a Cantor bouquet.
\end{observation}

\section{A model for strongly postcritically separated cosine maps}\label{sec_model_sps}

We defined in the introduction what it means for a cosine map to be strongly postcritically separated. This notion, introduced in \cite{mio_orbifolds}, generalizes to some functions in class $\B$:
\begin{defn}[Strongly postcritically separated functions]\label{def_strongps}
	A function $f\in \B$ is \textit{strongly postcritically separated (sps)} if there are constants $c,\epsilon >0$ such that
	\begin{enumerate}[label=(\alph*)]
		\item \label{item_Fatou} $P_{F}\defeq P(f)\cap F(f)$ is compact;
		\item \label{itema_defsps} $f$ has bounded criticality on $J(f)$;
		\item \label{itemb_defsps} for each $z\in J(f)$, $\#(\Orb^+(z)\cap \Crit(f)) \leq c$;
		\item \label{itemd_defsps} for all distinct $z,w\in P_J\defeq P(f)\cap J(f)$, $\vert z-w\vert\geq \epsilon \max\{\vert z \vert, \vert w \vert\}$.
	\end{enumerate}
\end{defn}
For cosine maps, some conditions in the definition of strongly postcritically separated maps are trivially satisfied, and thus they can be characterized the following way:
\begin{prop}[Cosine maps that are strongly postcritically separated]\label{prop_sps_cosine} Let $f$ be in the cosine family. Then the following are equivalent:
	\begin{enumerate}[label=(\Alph*)]
		\item \label{item_equiv_strong1} $f$ is strongly postcritically separated;
		\item \label{item_equiv_strong2} $P_F$ is compact and there exists $\epsilon>0$ such that for all distinct $z,w\in P_J$, 
		\begin{equation}\label{eq_epsilon} 
		\vert z-w\vert\geq \epsilon \max\{\vert z \vert, \vert w \vert\}.
		\end{equation}
		\item \label{item_equiv_strong3} Each critical value of $f$ converges to an attracting cycle or a repelling periodic cycle, or its orbit tends to infinity in such a way that \eqref{eq_epsilon} holds. 
	\end{enumerate}
\end{prop}
\begin{proof}
	By definition, \ref{item_equiv_strong1} $\Rightarrow$ \ref{item_equiv_strong2}. For any cosine map $f$, $\AV(f)=\emptyset$ and any critical point has local degree equal to $2$, see \ref{dis_basic_cosine}. In particular, $f$ has bounded criticality on its Julia set. Moreover, since each cosine map has two critical values, for all $z\in J(f)$, $\#(\Orb^+(z)\cap \Crit(f)) \leq 2$. Hence, if $f$ is in the cosine family and \ref{item_equiv_strong2} holds for $f$, then all conditions in the definition of strongly postcritically separated maps (Definition \ref{def_strongps}) are satisfied, and so \ref{item_equiv_strong2} $\Rightarrow$ \ref{item_equiv_strong1}. If \eqref{eq_epsilon} holds for $f$, then $P(f)\cap J(f)$ is discrete. In addition, since $f\in \B$, when $P_J$ is discrete, $P_F$ being compact is equivalent to all periodic cycles in $J(f)$ being repelling and $F(f)$ being a collection of attracting basins (see the proof of \cite[Lemma 2.6]{mio_orbifolds}), and thus, \ref{item_equiv_strong2} $\Leftrightarrow$ \ref{item_equiv_strong3}.
\end{proof}

\begin{remark}
The definition that we gave in the introduction of sps cosine maps (Definition~\ref{def_sps_intro}) agrees with Proposition \ref{prop_sps_cosine}\ref{item_equiv_strong2}. 
\end{remark}

In this section we prove our main results on sps cosine maps; namely, Theorems \ref{thm_landing} and~\ref{thm_strong_cosine_intro}. We start by providing a formal definition of dynamic rays:
\begin{defn}[Dynamic rays for transcendental maps {\cite[Definition 2.2]{RRRS}}]\label{def_ray}
	Let $f$ be a transcendental entire function. A \emph{ray tail} of $f$ is an injective curve $\gamma :[t_0,\infty)\rightarrow I(f)$, with $t_0>0$, such that
	\begin{itemize}
		\item for each $n\geq 1$, $t \mapsto f^{n}(\gamma(t))$ is injective with $\lim_{t \rightarrow \infty} f^{n}(\gamma(t))=\infty$. 
		\item $f^{n}(\gamma(t))\rightarrow \infty$ uniformly 
		in $t$ as $n\rightarrow \infty$.
	\end{itemize}
	A \emph{dynamic ray} of $f$ is a maximal injective curve $\gamma :(0,\infty)\rightarrow I(f)$ such that the restriction $\gamma_{|[t,\infty)}$ is a ray tail for all $t > 0$. We say that $\gamma$ \emph{lands} at $z$ if $\lim_{t \rightarrow 0^+} \gamma(t)=z$, and we call $z$ the \emph{endpoint} of $\gamma$. We denote the set of endpoints of dynamic rays of $f$ by $E(f)$.
\end{defn}

Recall from the introduction that whenever a ray tail contains a critical value, its preimages can be interpreted as several ray tails that \textit{split} or \textit{break} at critical points, and that may be extended to overlap pairwise. In \cite{mio_signed_addr}, combinatorics for maps exhibiting this phenomenon are developed, where the extensions of ray tails at critical points, as described for the map $z \mapsto \cosh(z)$, are formalized.

More precisely, let $f$ be a cosine map, let  $\Addr(f)$ be the set of external addresses as specified in Definition \ref{def_addresses} and denote
\begin{equation}\label{eq_signedadd}
\Addr(f)_\pm\defeq  \Addr(f)\times \{-,+\}.
\end{equation}
We call each of its elements $(\s, \ast) \in \Addr(f)_\pm$ a \emph{signed (external) address}. 

It is shown in \cite[\S 3]{mio_signed_addr} that $\Addr(f)_\pm$ can be endowed with a topology such that each $z\in I(f)$ has at least two signed addresses that depend continuously on $z$. This leads to a folliation of $I(f)$ as a collection of rays, that we call \textit{canonical}, indexed by signed external addresses. More precisely, we shall use the following results from \cite{mio_signed_addr}.

\begin{dfn&thm}[Canonical rays]\label{defn_thm_canonical} \normalfont Let $f$ be a sps cosine map. Then, for each $(\s, \ast) \in \Addr(f)_\pm$  there is a curve $\Gamma(\s, \ast)$, that is either a ray tail or a dynamic ray possibly with its endpoint, such that
	\begin{equation}\label{eq_If_in Gamma}
I(f)\subset \bigcup_{(\s,\ast)\in \Addr(f)_\pm} \Gamma(\s, \ast).
	\end{equation}
We say that $\Gamma(\s, \ast)$ is a \textit{canonical ray}, and we call any ray tail contained in a canonical ray a \textit{canonical tail}. Landing of all canonical rays implies landing of all dynamic rays in $J(f)$. Each $z\in I(f)$ belongs to exactly 
\begin{equation}\label{eq_addrz}
\Addr(z)_\pm \defeq 2\prod^{\infty}_{j=0}\deg(f,f^{j}(z))
\end{equation}
different canonical rays, for $\Addr(z)_\pm \in \{2,4,8\}.$
\end{dfn&thm}

\begin{proof}
The definition and existence of canonical rays, as well as \eqref{eq_If_in Gamma}, are a consequence of \cite[Definition 3.5 and Theorem 3.8]{mio_signed_addr}, while landing of canonical rays implying landing of all dynamic rays is \cite[Observation 3.13]{mio_signed_addr}. Equation \eqref{eq_addrz} follows from \cite[Definition 3.10 and Observation 3.11]{mio_signed_addr}, and $\Addr(z)_\pm \in \{2,4,8\}$ is a consequence of cosine maps having only $2$ critical values, with all critical points being of local degree $2$.
\end{proof}

\begin{remark}It follows from the previous result that in order to study dynamic rays for the functions we consider, it suffices to study canonical rays.
\end{remark}

Next, we define our model for sps cosine maps. We start by setting the topology we shall use in the ambient space $\M\times \{-,+\}$ that contains $J(\F)\times \{-,+\}$:
\begin{discussion}[Definition of topologies]\label{dis_topology_spsmodel} Consider the set
	\begin{equation}\label{eq_Mpm}
	\M_\pm \defeq \M\times \{-,+\}= [0,\infty)\times(\Z\times\{L, R\})^\N \times \{-,+\}, 
	\end{equation}
	and let  $<_{\ell}$ denote the lexicographic order in $(\Z\times\{L, R\})^\N$ defined in \ref{dis_topology_model_cos}. Let us give the set $\lbrace -,+\rbrace$ the order $\lbrace -\rbrace \prec\lbrace +\rbrace$, and define the linear order in $(\Z\times\{L, R\})^\N \times \{-,+\}$
\begin{equation} \label{eq_linear_addr}
	(\ul{s}, \ast )<_{_A} (\ul{\tau}, \star) \qquad \text{ if and only if } \qquad \ul{s} <_{_\ell} \ul{\tau} \quad \text{ or } \quad \ul{s} =_{_\ell} \ul{\tau}\: \text{ and } \: \ast \prec \star,
\end{equation}
where the symbols ``$\ast, \star$'' denote generic elements of $\lbrace-, +\rbrace.$ 
This linear order gives rise to a cyclic order: for $a,x,b \in (\Z\times\{L, R\})^\N \times \{-,+\}$,
\begin{equation*} \label{eq_orderAfpm}
	[a,x,b]_{_A} \quad \text{if and only if} \quad a<_{_A} x<_{_A} b \quad \text{ or } \quad x <_{_A}b <_{_A} a \quad \text{ or }\quad b <_{_A} a <_{_A} x.
\end{equation*}
This cyclic order allows us to provide the set $(\Z\times\{L, R\})^\N \times \{-,+\}$ with the corresponding cyclic order topology $\tau_I$.

We define the topological space $(\M_\pm, \tau_{\M_\pm})$, with $\tau_{\M_\pm}$ being the product topology of $[0,\infty)$ with the usual topology, and $(\Z\times\{L, R\})^\N \times \{-,+\}$ with the topology $\tau_I$.
\end{discussion}

\begin{defn}[Model for sps cosine functions]\label{defn_modelspace_cos}
Let $J(\F)_\pm\defeq J(\F)\times \{-,+\}$ be the subspace of $\M_\pm$ with the induced topology $\tau_{\pm}$ from $\tau_{\M_\pm}$. We call $(J(\F)_\pm, \tau_{\pm})$ the \textit{model space} for strongly postcritically separated cosine maps. Define $I(\F)_{\pm}\defeq I(\F)\times \lbrace -,+\rbrace \subset J(\F)_{\pm}$ as a subspace equipped with the induced topology. We define the \textit{model function} as $\widetilde{\mathcal{F}}: J(\F)_\pm \to J(\F)_\pm$, given by $\widetilde{\F}(t,\s, \ast) \defeq(\F(t,\s), \ast)$. 
\end{defn}	



We are now ready to prove the main result of this paper, which is a more detailed version of Theorem \ref{thm_strong_cosine_intro}. Using the general framework developed in \cite{mio_splitting}, its proof will be concise. We remark that, alternatively, we could had provided a more direct but significantly more lengthy and technical proof. More precisely, one could avoid using the results from \cite{mio_splitting} and construct directly the semiconjugacy from $\F$ to $f\vert _{J(f)}$ in a similar manner to the proof of \cite[Theorem 6.5]{mio_splitting}, using that sps maps expand an orbifold metric in a neighbourhood of their Julia sets; \cite[Theorem 1.1]{mio_orbifolds}. We have opted for the first approach for clarity of exposition and to show an application of the framework that \cite{mio_splitting} provides.

Recall from Definition \ref{defn_modelcosine} that $\Exad$ denotes the set of exponentially bounded elements associated to $J(\F)$. We define $\Exadpm \defeq\Exad\times \{-,+\}$.
\begin{thm}\label{thm_strong_cosine} Let $f$ be a strongly postcritically separated cosine map and let $J(\F)_\pm$ and $\widetilde{\F}$ be the model space and function from Definition \ref{defn_modelspace_cos}. Then, $\widetilde{\F}$ is continuous, and there exists a continuous surjective function
\begin{equation*}
	\hat{\phi} \colon J(\F)_{\pm} \rightarrow J(f) \qquad \text{ so that } \qquad f\circ\hat{\phi} = \hat{\phi}\circ \widetilde{\F},
\end{equation*}
$\hat{\phi}(I(\F)_{\pm})=I(f)$ and such that for every $z\in I(f)$, $\# \hat{\phi}^{-1}(z)\in \{2,4,8\}$. Moreover, for each $(\s, \ast) \in \Exadpm$ the restriction map $\hat{\phi} \colon \{(t, \s, \ast) \colon t \geq t_\s\} \rightarrow \overline{\Gamma(\s, \ast)}$ is a bijection, and so $\overline{\Gamma(\s, \ast)}$ is a canonical ray together with its endpoint.
\end{thm}
\begin{remark} We have implicitly stated in Theorem \ref{thm_strong_cosine} that the map $\hat{\phi}$ establishes a one-to-one correspondence between $\Exadpm$ and $\Addr(f)_\pm$, since, with some abuse of notation, we have stated that for each $(\s, \ast) \in \Exadpm$, $\{(t, \s, \ast) \colon t \geq t_\s\}$ is mapped bijectively to $\overline{\Gamma(\s, \ast)} \subset J(f)$ for $(\s, \ast) \in \Addr(f)_\pm$. In particular, we are claiming that $\hat{\phi}$ is an order-preserving continuous map. 
\end{remark}	

\begin{proof}[Proof of Theorem \ref{thm_strong_cosine}] Let $g\defeq \lambda f$ be of disjoint type for some $\lambda\in \C^\ast$ and let $\Phi\colon J(\F)\to J(g)$ be the homeomorphism from Theorem \ref{thm_conjcosine_intro}. Define $J(g)_\pm\defeq J(g)\times \{-,+\}.$ In \cite[5.4, p. 23]{mio_splitting}, $J(g)_\pm$ is induced with a topology $\tau_J$ in an analogous manner as we did in \ref{dis_topology_spsmodel} for $J(\F)_\pm$, with $(\Z\times\{L, R\})^\N$ replaced by $\R\setminus \Q$ in \cite[equation (5.1)]{mio_splitting}. In particular, 
	\begin{equation} \label{eq_def_hatPhi}
	\hat{\Phi}\colon J(\F)_\pm\rightarrow J(g)_\pm; \quad \quad  (t,\s,\ast)\mapsto (\Phi(t, \s), \ast)
	\end{equation}
is a homeomorphism. We let $I(g)_\pm\defeq I(g)\times \{-,+\}$ and for each $\s\in \Addr(g)$, denote $J_{(\s, \ast)}\defeq J_\s\times \{\ast\}$. Let $\tilde{g}\colon J(g)_\pm\rightarrow J(g)_\pm$ be given by $\tilde{g}(z, \ast)\defeq(g(z), \ast).$  Note that by definition of the functions involved and Theorem \ref{thm_conjcosine_intro}, we have that
\begin{equation}\label{eq_commute_Ftilde}
\hat{\Phi}\circ \widetilde{\F}=\tilde{g} \circ \hat{\Phi}, 
\end{equation}
and, in particular, $\F$ is continuous as it can be expressed as a composition of continuous functions.

Then,  by \cite[Theorem 6.5]{mio_splitting}, there exists a continuous, surjective map $\phi: J(g)_\pm \to J(f)$ such that 
\begin{equation}\label{eq_comm_phi}
 f\circ\phi = \phi\circ \tilde{g}, \quad \quad  \phi(I(g)_{ \pm})=I(f)
\end{equation}
and for each $(\s, \ast)\in \Addr(g)_\pm$, 
\begin{equation}\label{eq_bijection}
\phi \colon J_{(\s, \ast)} \rightarrow \overline{\Gamma(\s, \ast)} \quad \text{is a bijection}, 
\end{equation}
where $\Gamma(\s, \ast)$ is the canonical ray at signed address $(\s, \ast)\in \Addr(f)_\pm$. It then follows from Definition and Theorem \ref{defn_thm_canonical} that for each $z\in I(f)$,
\begin{equation}\label{eq_cardinalAddr}
	\#\phi^{-1}(z)=2\prod^{\infty}_{j=0}\deg(f,f^{j}(z))\in \{2,4,8\}.
\end{equation}
	
Let $\hat{\phi}\defeq\phi\circ \hat{\Phi}: J(\F)_\pm \rightarrow J(f)$, which is continuous as it is a composition of continuous functions. Then, by  \eqref{eq_comm_phi} and \eqref{eq_commute_Ftilde},
$$f \circ \hat{\phi}=f \circ \phi \circ \hat{\Phi} =\phi\circ \tilde{g} \circ \hat{\Phi}= \phi\circ \hat{\Phi}\circ \widetilde{\F}= \hat{\phi} \circ \widetilde{\F},$$
as shown in the diagram:
\[\begin{tikzcd}
{\color{blue!80!black}J(\F)_\pm} \arrow[r,blue!80!black, "\widetilde{\F}"] \arrow[d,"\hat{\Phi}"]\arrow["\hat{\phi}"', dd, blue!80!black, bend right =55] & {\color{blue!80!black}J(\F)_\pm} \arrow[d, "\hat{\Phi}"] \arrow[dd, blue!80!black, bend left = 55, "\hat{\phi}"]\\
J(g)_{\pm} \arrow[r, "\tilde{g}"] \arrow[d, "\phi"]& J(g)_\pm \arrow[d, "\phi"] \\
{\color{blue!80!black}J(f)} \arrow[r, blue!80!black, "f"] & {\color{blue!80!black}J(f)}.	\end{tikzcd}
\]
Moreover, since, by Theorem \ref{thm_conjcosine_intro}, $\Phi(I(\F))=I(g)$, we have that $\hat{\Phi}(I(\F)_\pm)=I(g)_\pm$, and consequently $\hat{\phi}(I(\F)_\pm)=I(f)$. 
	 
Since $\hat{\Phi}$ is a homeomorphism, using  \eqref{eq_cardinalAddr} we have that $\#\hat{\phi}^{-1}(z)=\#\phi^{-1}(z)\in \{2,4,8\}$. By Observation~\ref{obs_model_bouquet}, $\Phi\colon \{(t, \s) \colon  t \geq t_\s\} \to J_\s$ is a bijection for each $\s\in \Exad$, with $J_\s$ being a landing dynamic ray. Hence, by \eqref{eq_def_hatPhi} and \eqref{eq_bijection}, $\hat{\phi} \colon \{(t, \s, \ast) \colon t \geq t_\s\} \rightarrow \overline{\Gamma(\s, \ast)}$ is also bijection and  $\overline{\Gamma(\s, \ast)}$ is a canonical ray together with its endpoint. 
\end{proof}


\begin{proof}[Proof of Theorem \ref{thm_landing}] Let $f$ be a sps cosine map. Then, by Definition and Theorem \ref{defn_thm_canonical}, proving that all canonical rays of $f$ land suffices to conclude that all its dynamic rays land. Since, by Theorem \ref{thm_strong_cosine}, for each $(\s, \ast) \in \Addr(f)_\pm$, $\overline{\Gamma(\s, \ast)}$ is a canonical ray together with its landing point, the result follows.  Moreover, by Theorem \ref{thm_strong_cosine}, 
$$J(f)=\hat{\phi}(J(\F)_\pm)= \hat{\phi}\left(\bigcup_{(\s, \ast) \in \Exadpm}\{(t, \s, \ast) \colon t \geq t_\s\}\right)= \bigcup_{(\s, \ast) \in \Addr(f)_\pm} \overline{\Gamma(\s,\ast)},$$
and so every point in $J(f)$ is either on a dynamic ray or it is the landing point of at least one such ray.
\end{proof}	
\section{Combinatorics and landing of rays for $f(z)=\cosh(z)$}\label{sec_cosh_indent}

In the previous section we provided a semiconjugacy between the restriction of any strongly postcritically separated cosine map $f$ to its Julia set, and a simple dynamical system formed by a map $\widetilde{\F}$ and a model space $J(\F)_\pm$, that reflects the splitting of dynamic rays at critical points. In particular, Theorem \ref{thm_strong_cosine} implies that $J(f)$ can be described as a collection of canonical rays that land, and overlap pairwise.

Given the explicit nature of cosine maps, we claim that it is possible to improve this result by providing a conjugacy rather than a semiconjugacy between a model dynamical system and our maps.  Note that if a cosine map $f$ is sps, we let $\hat{\phi} \colon J(\F)_{\pm} \rightarrow J(f)$ be the map from Theorem \ref{thm_strong_cosine}, and we define the equivalence relation in $J(\F)_\pm$
\begin{equation}\label{eq_sim}
a \sim b \iff \hat{\phi}(a)=\hat{\phi}(b),
\end{equation}
then, since $\hat{\phi}$ is continuous, by the Universal Property of Quotient Maps (see for example \cite[Theorem 2.22]{Munkres}), there exists a unique continuous function $\tilde{\phi}\colon J(\F)_\pm /{\sim} \rightarrow J(f)$ such that the diagram 
\[\begin{tikzcd}
J(\F)_\pm \arrow[r, "\hat{\phi}"] \arrow[d, "\pi" ] & J(f) \\
\faktor{J(\F)_\pm}{\sim} \arrow[ru, dashed, "\exists! \:\tilde{\phi}"']
\end{tikzcd}
\]
commutes, where $\pi$ is the projection function that takes each element to its equivalence class. In particular, since both $\hat{\phi}$ and $\pi$ are surjective, $\tilde{\phi}$ is by definition bijective. By the commutative relation $f\circ \hat{\phi}= \hat{\phi} \circ \widetilde{\F}$ from Theorem \ref{thm_strong_cosine}, for any $a,b\in J(\F)_\pm$, $$\pi(a)=\pi(b) \Rightarrow \hat{\phi}(\widetilde{\F}(a)) =f(\hat{\phi}(a))=f(\hat{\phi}(b))=\hat{\phi}(\widetilde{\F}(b))\Rightarrow \pi(\widetilde{\F}(a))=\pi(\widetilde{\F}(b)),$$
and so, the map $h\colon J(\F)_\pm /{\sim} \rightarrow J(\F)_\pm /{\sim}$ given by $h(\pi(x))\defeq \pi(\widetilde{\F}(x))$ is a well-defined homeomorphism. In particular, $\tilde{\phi}$ conjugates $h$ and $f$ as shown in the following diagram:
\[
\begin{tikzcd}
J(\F)_\pm \arrow[r, "\widetilde{\F}"] \arrow[d,"\pi"] \arrow["\hat{\phi}"' ,dd, bend right =55] & J(\F)_\pm\arrow[d, "\pi"] \arrow[dd, bend left = 55, "\hat{\phi}"]\\
{\color{blue!80!black} \faktor{J(\F)_\pm}{\sim}} \arrow[r, blue!80!black, "h"] \arrow[d, blue!80!black, "\tilde{\phi}"]&{\color{blue!80!black} \faktor{J(\F)_\pm}{\sim}}\arrow[d, blue!80!black, "\tilde{\phi}"] \\
{\color{blue!80!black}J(f)} \arrow[r, blue!80!black, "f"] & {\color{blue!80!black}J(f)}.
\end{tikzcd}
\]

In this section, we describe ``$\sim$'' explicitly for the maps $z\mapsto \cosh(z)$ and $z\mapsto \cosh^2(z).$ The following observation is important to us.

\begin{observation}\label{obs_equiv_overlapping} Recall from Theorem~\ref{thm_strong_cosine} that $\hat{\phi} (\{(t, \s, \ast) \colon t \geq t_\s\})= \overline{\Gamma(\s, \ast)}$ for each $(\s, \ast)\in \Addr(f)_\pm$. Hence, in order to describe $J(\F)_\pm /{\sim}$, we can equivalently determine the overlap occuring between the canonical rays  $\{\overline{\Gamma(\s, \ast)}\}_{(\s, \ast)\in \Addr(f)_\pm}$ with their endpoints.
\end{observation}	

We will divide our task into two: on one hand, we shall study whether canonical rays share their endpoints, that is, whether some of them land together. On the other hand, one should provide information on the overlaps occuring between canonical rays. For the second, we will use the following result, which is  \cite[Proposition 3.9]{mio_signed_addr}. We recall that $\sigma \colon \Addr(f)\rightarrow \Addr(f) $ denotes the one-sided shift map on addresses.
\begin{prop}[Overlapping of canonical rays] \label{prop_overlapping_Gamma} For each $(\s, \ast) \in \Addr(f)_\pm$, either $\Gamma(\s, -)=\Gamma(\s, +)$ when $\Orb^{-}(\Crit(f)) \cap \Gamma(\s, \ast) =\emptyset$, or $\Gamma(\s, \ast)$ can be expressed as a concatenation 
	\begin{equation}\label{def_concat2}
	\Gamma(\s, \ast)=\cdots \bm{\cdot} \{c_{i+1}\}\bm{\cdot} \gamma^{i+1}_{i} \bm{\cdot} \{c_i\} \bm{\cdot} \cdots \bm{\cdot}\gamma^{1}_{0} \bm{\cdot} \{c_0\} \bm{\cdot} \gamma^{\infty}_{c_0},
	\end{equation}
	where $\{c_i\}_{i\in I}=\Orb^{-}(\Crit(f)) \cap \Gamma(\s, \ast)$, for each $i\geq 1$, if it exists, the curve $\gamma^{i+1}_{i}$ is a (bounded) piece of dynamic ray, and $\gamma^{\infty}_{c_0}$ is a piece of dynamic ray joining $c_0$ to infinity. In particular, in the latter case, the following properties hold for $\Gamma(\s, \ast)$:
	\begin{enumerate}[label=(\Alph*)]
		\item \label{itemA_Gamma} $\gamma^{\infty}_{c_0} \cup \{c_0\}= \Gamma(\s,- ) \cap \Gamma(\s, +)$ and $\gamma^{\infty}_{c_0}$ does not belong to any other canonical ray.
		\item \label{itemB_Gamma} For each $i\geq 0$, the point $c_i$ belongs to exactly $2 \prod^{\infty}_{j=0}\deg(f,f^{j}(c_i))$ canonical rays.
		\item \label{itemC_Gamma} For each $i\geq 0$, $\gamma^{i+1}_i= \Gamma(\s, \ast)\cap \Gamma(\ultau, \star)$, where $\star\neq \ast$ and $\sigma^j(\ultau)=\sigma^j(\s)$ for some $j\geq 1$. Moreover, $\gamma^{i+1}_i$ does not belong to any other canonical ray.
	\end{enumerate}
\end{prop}

We start our analysis on the map $f(z)=\cosh(z)$. Following Observation~\ref{obs_equiv_overlapping}, we shall provide a combinatorial description of the equivalence classes of $J(\F)_\pm /{\sim}$ in terms of the signed addresses of their images under $\tilde{\phi}$. 

For a function $f\in \B$, the partition of a neighbourhood of infinity into fundamental domains is commonly regarded as a \textit{static partition} in the sense that the curve $\delta$ and domain $D\supset S(f)$ in its definition do not have dynamical meaning for $f$. In particular, dynamic rays of $f$ might cross the boundaries of fundamental domains infinitely often. Instead, for our specific example, we can define a \textit{dynamical partition}, so that the boundaries of the components are ray tails:

\begin{discussion}[Dynamical partition for $f(z)=\cosh(z)$]\label{dis_partition_cosh} For the function $f(z)\defeq \cosh(z)$, $J(f)=\C$ and $S(f)=\CV(f)=\{-1, 1\}$. Moreover, the curves $\gamma_1\defeq \R \setminus (-\infty, 1)$ and $\gamma_{-1}\defeq \R \setminus (-1, \infty)$ are ray tails joining $1$ and $-1$ to $\infty$ whose forward orbits lie in $\R^+$. Let 
\begin{equation}\label{eq_X}
X\defeq\gamma_1\cup \gamma_{-1}.
\end{equation}
Since $\C \setminus X$ is simply-connected and $S(f)\subset X$, by the Monodromy Theorem, each connected component of $f^{-1}(\C \setminus X)$ is a simply-connected domain, and the restriction of $f$ to it is a conformal isomorphism to its image. More specifically, noting that all preimages of the critical values of $f$ are critical points of local degree $2$, each connected component of $f^{-1}(\C \setminus X)$ is a horizontal strip of the form
	\begin{equation*}
	U_K\defeq \{ z\in \C \text{ such that } K\pi < \Ima z <(K+1)\pi) \}
	\end{equation*}
	for some $K\in\Z$. We denote $\mathcal{U}\defeq \{U_K\}_{K\in \Z}$; see Figure \ref{fig:rays_cosh}.
\end{discussion}

\begin{discussion}[Fixing signed addresses for $f(z)=\cosh(z)$]\label{dis_addr_cosh}
	Let us fix any bounded domain $D\supset [-1,1] \supset S(f)$. Then, $f^{-1}(\C\setminus D)$ consists of two unbounded domains that do not contain the imaginary axis, since $f(i\R)=[-1,1]$. Thus, we can choose $\delta\defeq i\R^+ \setminus D$ and define fundamental domains for $f$ as connected components of $f^{-1}(\C \setminus (\overline{D} \cup \delta))$. In particular, $f^{-1}(\delta)$ equals the collection of horizontal half-lines 
	\begin{equation*}
	\begin{split}
	\{z\in \C : \Rea z<0 & \text{ and }\Ima z= (-1/2+2n)\pi \},\\
	\{z\in \C: \Rea z >0 & \text{ and }\Ima z = (1/2+2 n)\pi \} 
	\end{split}
	\end{equation*}
	for all $n \in \Z$. Thus, each fundamental domain is contained in one of the half-strips
	\begin{equation}\label{fund_cosh_cos}
	\begin{split}
	S_{(n,L)}& \defeq \{z: \Rea z<0, \Ima z\in ((n-1/2)\pi ,(n+3/2)\pi)\},\\
	S_{(n,R)}& \defeq \{z: \Rea z>0, \Ima z\in ((n-3/2)\pi,(n+1/2)\pi)\}.
	\end{split}
	\end{equation}
	For each $(n,\ast)\in \Z \times \{-,+\}$, we denote by $n_\ast$ the unique fundamental domain contained in $S_{(n, \ast)}$. Using these fundamental domains, we define the set of external addresses $\Addr(f)$, and fix the corresponding set of signed external addresses $\Addr(f)_\pm$; see Definition and Theorem \ref{defn_thm_canonical}. In particular, for the curves $\gamma_1$ and $\gamma_{-1}$ introduced in \ref{dis_partition_cosh}, it holds that $f(\gamma_{-1}) \subset \gamma_1$, $f(\gamma_1)\subset \gamma_1$, $\gamma_1 \subset {\color{NavyBlue}0_R}$ and $\gamma_{-1}\subset {\color{RedViolet} 0_L}$. Moreover, each of these curves equals two canonical tails with opposite sign, as they do not contain preimages of critical points, see Proposition \ref{prop_overlapping_Gamma}. Hence, $\gamma_1\subset \Gamma({\color{NavyBlue}\cl{0_R}},\ast)$ and $\gamma_{-1}\subset \Gamma({\color{RedViolet} 0_L} {\color{NavyBlue}\cl{0_R}},\ast)$ for both $\ast \in \{-,+\}$; see Figure \ref{fig:fund_cosh}. 
\end{discussion}

\begin{figure}[htb]
\begingroup%
\makeatletter%
\providecommand\color[2][]{%
	\errmessage{(Inkscape) Color is used for the text in Inkscape, but the package 'color.sty' is not loaded}%
	\renewcommand\color[2][]{}%
}%
\providecommand\transparent[1]{%
	\errmessage{(Inkscape) Transparency is used (non-zero) for the text in Inkscape, but the package 'transparent.sty' is not loaded}%
	\renewcommand\transparent[1]{}%
}%
\providecommand\rotatebox[2]{#2}%
\ifx\svgwidth\undefined%
\setlength{\unitlength}{310.92519531bp}%
\ifx\svgscale\undefined%
\relax%
\else%
\setlength{\unitlength}{\unitlength * \real{\svgscale}}%
\fi%
\else%
\setlength{\unitlength}{\svgwidth}%
\fi%
\global\let\svgwidth\undefined%
\global\let\svgscale\undefined%
\makeatother%
\begin{picture}(1,0.89527067)%
\put(0,0){\includegraphics[width=\unitlength,page=1]{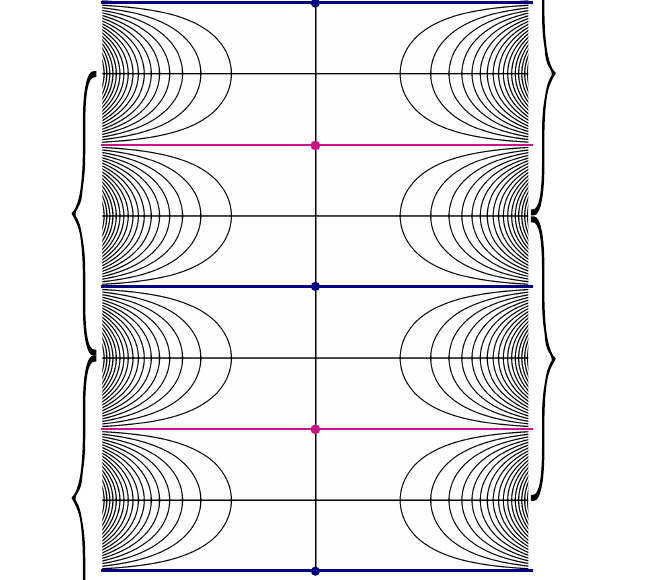}}%
\put(0.87259648,0.77319279){\color[rgb]{0,0,0.50196078}\makebox(0,0)[lb]{\smash{\fontsize{9pt}{1em}$1_R$}}}%
\put(0.49841163,0.42299278){\color[rgb]{0,0,0.50196078}\makebox(0,0)[lb]{\smash{\fontsize{9pt}{1em}$0$}}}%
\put(0.49682624,0.64019487){\color[rgb]{0.78039216,0.08235294,0.52156863}\makebox(0,0)[lb]{\smash{\fontsize{9pt}{1em}$\pi i$}}}%
\put(0.4952408,0.86215321){\color[rgb]{0,0,0.50196078}\makebox(0,0)[lb]{\smash{\fontsize{9pt}{1em}$2\pi i$}}}%
\put(0.48575493,0.02421196){\color[rgb]{0,0,0.50196078}\makebox(0,0)[lb]{\smash{\fontsize{9pt}{1em}$-2\pi i$}}}%
\put(0.48734037,0.19709332){\color[rgb]{0.78039216,0.08235294,0.52156863}\makebox(0,0)[lb]{\smash{\fontsize{9pt}{1em}$-\pi i$}}}%
\put(0.86680642,0.33018975){\color[rgb]{0,0,0.50196078}\makebox(0,0)[lb]{\smash{\fontsize{9pt}{1em}$0_R$}}}%
\put(0.04825316,0.10817699){\color[rgb]{0.78039216,0.08235294,0.52156863}\makebox(0,0)[lb]{\smash{\fontsize{9pt}{1em}$-1_L$}}}%
\put(0.05979315,0.54886508){\color[rgb]{0.78039216,0.08235294,0.52156863}\makebox(0,0)[lb]{\smash{\fontsize{9pt}{1em}$0_L$}}}%
\end{picture}%
\endgroup%

\caption[ASD]{\label{fig:fund_cosh} Partition of the plane into fundamental domains and itinerary components of $f(z)=\cosh(z)$. Each strip of height $\pi$ between two coloured horizontal lines is an itinerary domain. Strips that contain some fundamental domains are indicated by keys. Black horizontal lines are preimages of the imaginary axis, and the rest of curves are the preimages of all horizontal lines.}
\end{figure}

The dynamical partition from \ref{dis_partition_cosh} will help us determine that no two dynamic rays of $f$ land together. This is because since the boundaries of the elements in $\mathcal{U}$ are ray tails, as we shall see, no other ray tails can \textit{cross} them. More precisely, in the next proposition, we assign to each curve $\{\overline{\Gamma(\s, \ast)}\}_{(\s,\ast)\in \Addr(f)_\pm}$ a unique element of $\mathcal{U}$:
\begin{prop}[Each canonical ray is in the closure of a unique $U\in \mathcal{U}$]\label{prop_itin_cosh} For each $(\s, \ast)\in \Addr(f)_\pm$, there exists a unique component $U\in \mathcal{U}$ such that $\overline{\Gamma(\s, \ast)} \subset \cl{U}$. We denote 
	\[U(\s, \ast)\defeq U.\]
\end{prop}
\begin{proof}
	Since, as described in \ref{dis_addr_cosh}, the set $X$ defined in \eqref{eq_X} consists of four canonical tails that overlap pairwise, and, in addition, $f$ is totally ramified, exactly four canonical tails meet at each critical point of $f^{-1}(X)$, and their union is a connected component of this set. More precisely, following the analysis in \ref{dis_addr_cosh}, $f^{-1}(\gamma_1)$ is the collection of all the horizontal lines $\{2\pi K i\R\}_{K \in \Z}$, that in particular contain the critical points $2\pi K i$ for all $K\in \Z$. Analogously, $f^{-1}(\gamma_{-1})$ is the collection of the horizontal lines $\{2(K+1)\pi i \R\}_{K \in \Z}$ with critical points at $2(K+1)\pi i$ for each $K\in \Z$. Hence, it follows that for each $K\in \Z$ and $\ast\in \{-,+\}$, 
	\begin{align}\label{eq_boundariesUk_cosh}
	2\pi K\R^- & \subset \Gamma({\color{RedViolet} K_L} {\color{NavyBlue}\cl{0_R}},\ast), & (2K+1)\pi\R^- & \subset \Gamma({\color{RedViolet} K_L}{\color{RedViolet} 0_L}{\color{NavyBlue}\cl{0_R}},\ast), \\ 
	2\pi K\R^+ & \subset \Gamma({\color{NavyBlue} K_R \cl{0_R}}, \ast), & (2K+1)\pi\R^+ &\subset \Gamma({\color{NavyBlue} (K+1)_R} {\color{RedViolet} 0_L} {\color{NavyBlue}\cl{0_R}}, \ast).\nonumber
	\end{align}
	We claim that each of the canonical rays displayed in \eqref{eq_boundariesUk_cosh} belongs to the closure of exactly one component of $\mathcal{U}$: to see this, let us consider the curves $\Gamma({\color{NavyBlue} K_R \cl{0_R}}, \ast)$ for some $K\in \Z$ and both $\ast\in \{-,+\}$. Then, since $\Gamma({\color{NavyBlue} K_R \cl{0_R}}, -)$ is a nested sequence of left-extended canonical tails, see \cite[Definition 3.5 and Theorem 3.8]{mio_signed_addr}, and by Proposition \ref{prop_overlapping_Gamma} it can only intersect the boundaries of the elements of $\mathcal{U}$ in the subcurve $2K\pi i \R^+$, we conclude that $\Gamma({\color{NavyBlue} K_R \cl{0_R}}, -) \subset \overline{U_{K-1}}$. Similarly, $\Gamma({\color{NavyBlue} K_R \cl{0_R}}, +) \subset \overline{U_{K}}$, and arguing analogously for the rest of curves in \eqref{eq_boundariesUk_cosh}, the claim follows. Since no other canonical rays apart from those in \eqref{eq_boundariesUk_cosh} intersect $\C \setminus \mathcal{U}$, each of them must be contained in a unique component $U$, as stated.
\end{proof}
\begin{defn}[Itineraries for $f(z)=\cosh(z)$]
	For each $(\s, \ast)\in \Addr(f)_\pm$, we define the \textit{itinerary} of $(\s, \ast)$ as the infinite sequence
	$$\itin(\s, \ast)\defeq U(\s, \ast) U(\sigma(\s), \ast)U(\sigma^2(\s), \ast)\ldots. $$
\end{defn}

\begin{observation}[Itineraries of points]\label{obs_itin_points} Since for each $(\s, \ast) \in \Addr(f)_\pm$, $f(\overline{\Gamma(\s, \ast)})\subset \overline{\Gamma}(\sigma(\s), \ast)$, if $z\in \overline{\Gamma(\s, \ast)}$ and $\itin(\s, \ast)=U_0U_1\ldots$, then $f^i(z)\subset U_i$ for all $i\geq 0$.
\end{observation}

\begin{prop}[Dynamic rays of $z \mapsto \cosh(z)$ do not land together]\label{prop_landing_cosh}
	There are no two dynamic rays of $z \mapsto \cosh(z)$ landing together. 	
\end{prop}
\begin{proof} It suffices to show that there are no two canonical rays landing together, since then, by  Definition and Theorem \ref{defn_thm_canonical}, no two rays would land together. With that aim, let $\Gamma(\s, \ast)$ and $\Gamma(\ultau, \star)$ be two different canonical rays, that is, $(\s, \ast)\neq (\ultau, \star)$, and let $p_{(\s, \ast)}$ and $p_{(\ultau, \star)}$ be their respective endpoints. If $\Gamma(\s, \ast)$ and $\Gamma(\ultau, \star)$ land together, i.e., $p_{(\s, \ast)}= p_{(\ultau, \star)}$, then by Proposition \ref{prop_itin_cosh} and Observation \ref{obs_itin_points}, $\itin(\s, \ast)=\itin(\ultau, \ast)=U_0U_1\ldots$. Moreover, for each $i\geq 0$, $f^i(p_{(\s, \ast)})=f^i(p_{(\ultau, \star)})$ must belong to the interior of $U_i$, since by \ref{dis_addr_cosh}, the  boundaries of the elements of $\mathcal{U}$ are formed by canonical tails that are contained in dynamic rays. For the same reason, $i\R^+$ and $f^{-1}(i\R^+)$ do not contain any endpoints of dynamic rays, as they are formed by pieces of ray tails; see \ref{dis_overlapping_cosh} for more details. Then, for each $i\geq 0$, $f^i(p_{(\s, \ast)})=f^i(p_{(\ultau, \star)})$ belongs to a half-strip of the form 
	\begin{equation}\label{eq_4strips}
	\begin{split}
	HS^i_{(k,L)}&\defeq \left\{ z: \Rea z\leq 0, \Ima z\in \left(\frac{(i+4k)\pi}{4} ,\frac{(i+1+4k)\pi}{4}\right) \right\},\\
	HS^i_{(k,R)}&\defeq \left\{ z: \Rea z\geq 0, \Ima z\in \left(\frac{(i+4k)\pi}{4} ,\frac{(i+1+4k)\pi}{4}\right) \right\}
	\end{split}
	\end{equation}
	for some $k\in \Z$ and $0\leq i \leq 3$. However, each of the half-strips in \eqref{eq_4strips} intersects a single fundamental domain, see \eqref{fund_cosh_cos} and Figure \ref{fig:fund_cosh}, which contradicts $(\s, \ast)\neq (\ultau, \star)$.
\end{proof} 

\noindent Finally, we provide a combinatorial description of the overlaps that occur between canonical rays in terms of their signed addresses, which by Observation \ref{obs_equiv_overlapping} and Proposition \ref{prop_landing_cosh} suffices to describe the equivalence classes in $J(\F)/ {\sim}$.

\begin{discussion}[Overlapping of canonical tails for $ z \mapsto \cosh(z)$]\label{dis_overlapping_cosh}
	Recall that by Proposition \ref{prop_overlapping_Gamma}, for all $(\s, \ast)\in \Addr(f)_\pm$ such that $\Gamma(\s, \ast) \cap \Orb^-(\Crit (f))=\emptyset$,
	$\Gamma(\s, -)=\Gamma(\s, +).$ Hence, all other overlap occurs between the preimages of the canonical tails that contain $\Crit(f)$. Recall from \ref{dis_partition_cosh} and \ref{dis_addr_cosh} that $\Crit(f)=\{\pi i K: K \in \Z\}$, and each critical point belongs exactly to four canonical rays. Namely, we saw in \eqref{eq_boundariesUk_cosh} that
	\begin{align*}
	\Gamma({\color{RedViolet} K_L} {\color{NavyBlue}\cl{0_R}},-)&= \Gamma({\color{RedViolet} K_L} {\color{NavyBlue}\cl{0_R}},+) &&\text{in } 2K\pi i\R^- &&\text{ for all } K \in \Z. \\
	\Gamma({\color{NavyBlue} K_R \cl{0_R}}, -) &= \Gamma({\color{NavyBlue} K_R \cl{0_R}}, +) &&\text{in } 2K\pi i\R^+ &&\text{ for all } K \in \Z.\\
	\Gamma({\color{RedViolet} K_L}{\color{RedViolet} 0_L}{\color{NavyBlue}\cl{0_R}},-) &= \Gamma({\color{RedViolet} K_L}{\color{RedViolet} 0_L}{\color{NavyBlue}\cl{0_R}},+) &&\text{in } (2K+1)\pi i\R^{-} &&\text{ for all } K \in \Z. \\
	\Gamma({\color{NavyBlue} (K+1)_R} {\color{RedViolet} 0_L} {\color{NavyBlue}\cl{0_R}}, -) &= \Gamma({\color{NavyBlue} (K+1)_R} {\color{RedViolet} 0_L} {\color{NavyBlue}\cl{0_R}}, +) &&\text{in } (2K+1)\pi i\R^{+} &&\text{ for all } K \in \Z.
	\end{align*}
	
\begin{figure}[htb]
		\centering
		\resizebox{0.6
			\textwidth}{!}{\begingroup%
			\makeatletter%
			\providecommand\color[2][]{%
				\errmessage{(Inkscape) Color is used for the text in Inkscape, but the package 'color.sty' is not loaded}%
				\renewcommand\color[2][]{}%
			}%
			\providecommand\transparent[1]{%
				\errmessage{(Inkscape) Transparency is used (non-zero) for the text in Inkscape, but the package 'transparent.sty' is not loaded}%
				\renewcommand\transparent[1]{}%
			}%
			\providecommand\rotatebox[2]{#2}%
			\ifx\svgwidth\undefined%
			\setlength{\unitlength}{368.50393066bp}%
			\ifx\svgscale\undefined%
			\relax%
			\else%
			\setlength{\unitlength}{\unitlength * \real{\svgscale}}%
			\fi%
			\else%
			\setlength{\unitlength}{\svgwidth}%
			\fi%
			\global\let\svgwidth\undefined%
			\global\let\svgscale\undefined%
			\makeatother%
			\begin{picture}(1,0.69230772)%
			\put(0,0){\includegraphics[width=\unitlength]{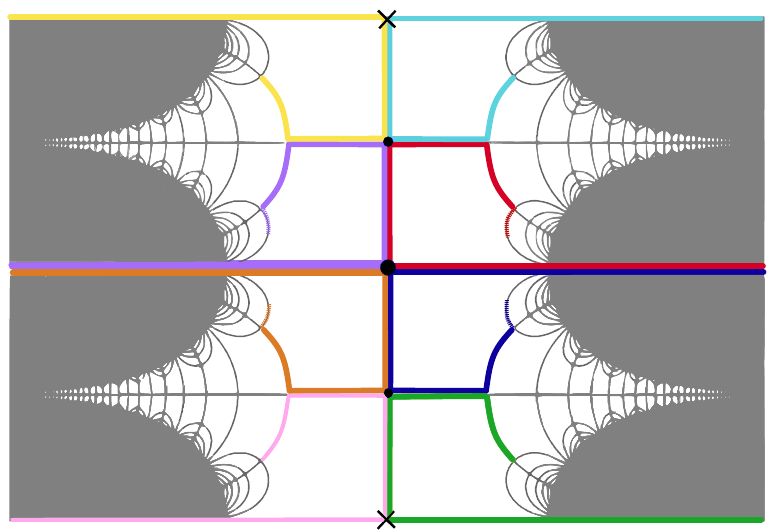}}%
			\put(0.51311181,0.21003289){\color[rgb]{0,0,0}\makebox(0,0)[lb]{\smash{$\fontsize{9pt}{1em} -i\frac{\pi}{2}$}}}%
			\put(0.51591789,0.35817888){\color[rgb]{0,0,0}\makebox(0,0)[lb]{\smash{$0$}}}%
			\put(0.51520385,0.52821125){\color[rgb]{0,0,0}\makebox(0,0)[lb]{\smash{$\fontsize{9pt}{1em} i\frac{\pi}{2}$}}}%
			\end{picture}%
			\endgroup%
		}
		\caption{Some canonical tails in the Julia set of the map $z \mapsto \cosh(z)$ that belong to canonical rays. Colour code: the red tail is in $\Gamma({\color{NavyBlue}\cl{0_R}}, + )$, the purple in $\Gamma({\color{RedViolet} 0_L} {\color{NavyBlue}\cl{0R}},-)$, the orange in $\Gamma({\color{RedViolet} 0_L} {\color{NavyBlue}\cl{0_R}},+)$ and the dark blue one in $\Gamma({\color{NavyBlue}\cl{0_R}}, -)$. Then, the light blue tail is in $\Gamma({\color{NavyBlue}1_R}{\color{RedViolet} 0_L} {\color{NavyBlue}\cl{0_R}},-)$, the yellow one in $\Gamma({\color{RedViolet} 0_L 0_L} {\color{NavyBlue}\cl{0_R}},+)$, the green is in $\Gamma({\color{NavyBlue}0_R}{\color{RedViolet} 0_L} {\color{NavyBlue}\cl{0_R}},+)$ and the pink is in $\Gamma({\color{RedViolet} -1_L 0_L}{\color{NavyBlue}\cl{0_R}},-)$.}
		\label{fig:rays_cosh}
	\end{figure}

Further identifications between the canonical rays above occur at the connected components of $f^{-1}([0,1])$ and $f^{-1}([-1,0])$. More precisely, if 
	for each $\pm\in \{-,+\}$ we denote 
	$$V_K(\pm)\defeq \left\{z \in \C: \Rea z=0 \text{ and } \Ima z\in [K\pi i,(K \pm 1/2)\pi i ]\right\},$$
	then we have the following identifications:
	\begin{align*}
	\Gamma({\color{RedViolet} K_L} {\color{NavyBlue}\cl{0_R}},\mp) &= \Gamma({\color{NavyBlue} K_R \cl{0_R}}, \pm) \qquad &&\text{in } V_{2K}(\pm) &&\text{for all } K \in \Z. \\
	\Gamma({\color{RedViolet} K_L}{\color{RedViolet} 0_L} {\color{NavyBlue}\cl{0_R}},\mp) & = \Gamma({\color{NavyBlue} (K+1)_R} {\color{RedViolet} 0_L} {\color{NavyBlue}\cl{0_R}}, \pm) &&\text{in } V_{2K+1}(\pm) &&\text{for all } K \in \Z, 
	\end{align*} 
	where $\mp=+$ when $\pm=-$, and $\mp=-$ when $\pm=+$; see Figure \ref{fig:rays_cosh}. By Proposition \ref{prop_overlapping_Gamma}, any further overlap between canonical rays occur at preimages of the overlap already stated. More specifically, since we have already described all overlap occurring at the boundaries of the elements in $\mathcal{U}$, all remaining ones must occur between canonical rays whose itinerary agrees on the first $N$-th elements and that are mapped under $f^N$ to the coordinate axes for some $N\in \N$. Given the geometry of the fundamental domains, in particular contained in half-strips of height $\pi$, see \eqref{fund_cosh_cos} and Figure \ref{fig:fund_cosh}, providing the identifications at the preimages of the positive imaginary axis allows us to express identifications solely using external addresses. This is because any further identifications must occur at the intersection of a fundamental domain with a component of $\mathcal{U}$, and hence, the corresponding first entries on the signed addresses of rays that overlap are the same. More specifically, for each $\pm\in \{-,+\}$, let us denote
	$$I^K_{{\color{NavyBlue} P_R}}(\pm)\defeq f^{-1}(V_K(\pm)) \cap (P+\pi i/2) \R^+ \quad \text{ and } \quad I^K_{{\color{RedViolet} P_L}}(\pm)\defeq f^{-1}(V_K(\pm)) \cap (P+3\pi i/2) \R^-.$$ Then, for all $K \in \Z^+ \text{ and } P\in \Z$, 
	\begin{align*}
	\Gamma({\color{NavyBlue}(P+1)_R}{\color{RedViolet} K_L} {\color{NavyBlue}\cl{0_R}},\mp) &= \Gamma({\color{NavyBlue} P_R K_R \cl{0_R}}, \pm) &&\text{in } I^{2K}_{{\color{NavyBlue} P_R}}(\pm) \\
	\Gamma({\color{NavyBlue}(P+1)_R}{\color{RedViolet} K_L}{\color{RedViolet} 0_L} {\color{NavyBlue}\cl{0_R}},\mp) & = \Gamma({\color{NavyBlue} P_R (K+1)_R} {\color{RedViolet} 0_L} {\color{NavyBlue}\cl{0_R}}, \pm) &&\text{in } I^{2K+1}_{{\color{NavyBlue} P_R}}(\pm)\\ 
	\Gamma({\color{RedViolet}P_L}{\color{RedViolet} K_L} {\color{NavyBlue}\cl{0_R}},\mp) &= \Gamma({\color{RedViolet}(P+1)_L}{\color{NavyBlue} K_R \cl{0_R}}, \pm) && \text{in } I^K_{{\color{RedViolet} P_L}}(\pm) \\
	\Gamma({\color{RedViolet}P_L}{\color{RedViolet} K_L}{\color{RedViolet} 0_L} {\color{NavyBlue}\cl{0_R}},\mp) & = \Gamma({\color{RedViolet}(P+1)_L}{\color{NavyBlue} (K+1)_R} {\color{RedViolet} 0_L} {\color{NavyBlue}\cl{0_R}}, \pm) &&\text{in } I^{2K+1}_{{\color{RedViolet} P_L}}(\pm),
	\end{align*}
	where $\mp=+$ when $\pm=-$, and $\mp=-$ when $\pm=+$; see Figure \ref{fig:rays_cosh}. Moreover, for all other canonical rays, if $\delta$ is some bounded curve in $J(f)$, then
	\begin{equation}\label{eq_further_overlap}
	\Gamma(\s, \ast) =\Gamma(\ul{\tau}, \star) \text { in } \delta \iff \renewcommand{\arraystretch}{1.5}\left\{\begin{array}{@{}l@{\quad}l@{}}
	\exists \, n>0 : s_j=\tau_j \text{ for all } j\leq n \quad \text{and}\\
	\Gamma(\sigma^n(\s), \ast) = \Gamma(\sigma^n(\ul{\tau}), \star) \text { in } f^n(\delta).
	\end{array}\right.\kern-\nulldelimiterspace
	\end{equation}
\end{discussion}

\begin{figure}[htb]
	\centering
	\begingroup%
	\makeatletter%
	\providecommand\color[2][]{%
		\errmessage{(Inkscape) Color is used for the text in Inkscape, but the package 'color.sty' is not loaded}%
		\renewcommand\color[2][]{}%
	}%
	\providecommand\transparent[1]{%
		\errmessage{(Inkscape) Transparency is used (non-zero) for the text in Inkscape, but the package 'transparent.sty' is not loaded}%
		\renewcommand\transparent[1]{}%
	}%
	\providecommand\rotatebox[2]{#2}%
	\ifx\svgwidth\undefined%
	\setlength{\unitlength}{311.81103516bp}%
	\ifx\svgscale\undefined%
	\relax%
	\else%
	\setlength{\unitlength}{\unitlength * \real{\svgscale}}%
	\fi%
	\else%
	\setlength{\unitlength}{\svgwidth}%
	\fi%
	\global\let\svgwidth\undefined%
	\global\let\svgscale\undefined%
	\makeatother%
	\begin{picture}(1,0.80909082)%
	\put(0,0){\includegraphics[width=\unitlength]{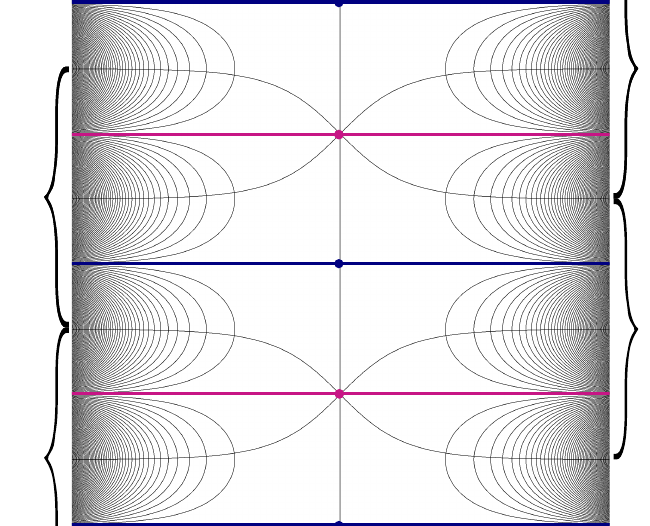}}%
	\put(0.99233345,0.69532413){\color[rgb]{0,0,0.50196078}\makebox(0,0)[lb]{\smash{\fontsize{9pt}{1em}$1_R$}}}%
	\put(0.53323567,0.37316004){\color[rgb]{0,0,0.50196078}\makebox(0,0)[lb]{\smash{\fontsize{9pt}{1em}$0$}}}%
	\put(0.53007391,0.77667498){\color[rgb]{0,0,0.50196078}\makebox(0,0)[lb]{\smash{\fontsize{9pt}{1em}$2\pi i$}}}%
	\put(0.52061497,0.01086533){\color[rgb]{0,0,0.50196078}\makebox(0,0)[lb]{\smash{\fontsize{9pt}{1em}$-2\pi i$}}}%
	\put(0.52287156,0.16749591){\color[rgb]{0.78039216,0.08235294,0.52156863}\makebox(0,0)[lb]{\smash{\fontsize{9pt}{1em}$-\pi i$}}}%
	\put(0.9922573,0.28510403){\color[rgb]{0,0,0.50196078}\makebox(0,0)[lb]{\smash{\fontsize{9pt}{1em}$0_R$}}}%
	\put(0.00027239,0.09936939){\color[rgb]{0.78039216,0.08235294,0.52156863}\makebox(0,0)[lb]{\smash{\fontsize{9pt}{1em}$-1_L$}}}%
	\put(0.0259702,0.49676373){\color[rgb]{0.78039216,0.08235294,0.52156863}\makebox(0,0)[lb]{\smash{\fontsize{9pt}{1em}$0_L$}}}%
	\put(0.54216531,0.56099235){\color[rgb]{0.78039216,0.08235294,0.52156863}\makebox(0,0)[lb]{\smash{\fontsize{9pt}{1em}$\pi i$}}}%
	\put(0.52287156,0.16749591){\color[rgb]{0.78039216,0.08235294,0.52156863}\makebox(0,0)[lb]{\smash{\fontsize{9pt}{1em}$-\pi i$}}}%
	\put(0.53007391,0.77667498){\color[rgb]{0,0,0.50196078}\makebox(0,0)[lb]{\smash{\fontsize{9pt}{1em}$2\pi i$}}}%
	\put(0.53323567,0.37316004){\color[rgb]{0,0,0.50196078}\makebox(0,0)[lb]{\smash{\fontsize{9pt}{1em}$0$}}}%
	\put(0.52061497,0.01086533){\color[rgb]{0,0,0.50196078}\makebox(0,0)[lb]{\smash{\fontsize{9pt}{1em}$-2\pi i$}}}%
	\end{picture}%
	\endgroup%
	
	\caption[as]{Partition of the plane into fundamental domains and itinerary components for $z \mapsto \cosh^2(z)$. Each strip of height $\pi$ between two coloured lines is an itinerary domain. Some fundamental domains are indicated by keys. Also displayed are the first (coloured lines and imaginary axis), second (other curves that meet at $\{K\pi i: K\in \Z\}$) and third (rest of curves) iterated preimages of the real line.}
	\label{fig:fund_cosh222}
\end{figure}

We have now provided a combinatorial description of all overlappings occurring between canonical rays of $z \mapsto \cosh(z)$, as well as determined that no two of its dynamic rays land together. Hence, we have concluded our analysis.

\begin{example}[Overlapping for the map $z \mapsto \cosh^2(z)$]
	Seeking an example where a critical value is mapped to another critical value, we consider the function 
	$$f(z)\defeq \cosh^2 (z)\defeq \cosh (z) \cdot \cosh (z) =\dfrac{e^{2z}+ e^{-2z}}{4}+ \dfrac{1}{2}.$$
	Even if, strictly speaking, this function is not in the cosine family, it is in the same parameter space, since $f$ is conjugate to $\frac{e}{2}e^w+ \frac{e^{-1}}{2}e^{-w}$ via $z \mapsto 2z-1$. The dynamics of $f$ have already been explored, namely in \cite{ripponFast}, where it is shown that $I(f)$ (and in fact its \textit{fast escaping set}) is connected. The function $f$ is $\pi i$-periodic, and $S(f)= \CV(f)=\{0,1\}$, with $f(0)=1 \in I(f)$ and $\Orb^+(1) \subset \R^+$. The critical points of $f$ are $f^{-1}(1)=\{K\pi i : K \in \Z \}$ and $f^{-1}(0)=\{(K+1/2)\pi i : K \in \Z \}.$ As for the map $z \mapsto \cosh(z)$, we can join the critical values to infinity using the ray tails $\gamma_1\defeq \R \setminus (-\infty, 1)$ and $\gamma_{0}\defeq \R \setminus (0, \infty)$. We define a dynamical partition for $f$ as the union of the connected components of $\C \setminus f^{-1}(\gamma_0\cup \gamma_1)$, which are horizontal half-strips of $\pi/2$-height that we call \textit{itinerary components}; see Figure \ref{fig:fund_cosh222}. By choosing a bounded domain $D$ containing $[0,1]$ and $\delta\defeq i\R^+\setminus D$, we can define fundamental domains for $f$ as the connected components of $f^{-1}(\C\setminus (\overline{D}\cup \delta))$. In particular, we label as ${\color{NavyBlue} 0_R}$ the component that contains an unbounded subset of $\R^+$, and as ${\color{RedViolet} 0_L}$ the component that contains an unbounded subcurve of $\R^-$. Additionally, we label as ${\color{NavyBlue} K_R}$ and ${\color{RedViolet} K_L}$ their respective $K\pi i$-translates; see Figure \ref{fig:fund_cosh222}. With this notation, the following identifications between canonical rays occur:
	\begin{align*}
	\Gamma({\color{NavyBlue} K_R \cl{0_R}}, -) &= \Gamma({\color{NavyBlue} K_R \cl{0_R}}, +) &&\text{in } K\pi i \R^+ &&\text{ for all } K \in \Z. \\
	\Gamma({\color{RedViolet} K_L} {\color{NavyBlue}\cl{0_R}},-) &= \Gamma({\color{RedViolet} K_L} {\color{NavyBlue}\cl{0_R}},+) &&\text{in } K\pi i\R^{-} &&\text{ for all } K \in \Z. 
	\end{align*}
	
	The main difference between the overlap on the canonical rays of $z \mapsto\cosh(z)$ and the overlap on those of $f$, is that since the critical points in $f^{-1}(0)$ are mapped to a critical point, each of them belongs to eight canonical tails rather than to four, and moreover, both singular values belong to the canonical rays $\Gamma({\color{NavyBlue} K_R \cl{0_R}}, \ast)$ for both $\ast\in \{-,+\}$. Compare Figures \ref{fig:rays_cosh} and \ref{fig:rays_cosh222}. Then, we further have the identifications
	\begin{align*}
	\Gamma({\color{RedViolet} K_L} {\color{NavyBlue}\cl{0_R}},\mp) &= \Gamma({\color{NavyBlue} K_R \cl{0_R}}, \pm) &&\text{in } [K\pi i,(1 \pm 1/2)K\pi i ]\eqdef V_K(\pm) &&\text{for all } K \in \Z,
	\end{align*}
	where $\mp=+$ when $\pm=-$, and $\mp=-$ when $\pm=+$. If for each $\pm \in \{-,+\}$ and $K\in \Z$ we let
	$$I^K_{{\color{NavyBlue} P_R}}(\pm)\defeq f^{-1}(V_K(\pm)) \cap{\color{NavyBlue} \overline{P_R} \cap \overline{(P+1)_R}  } \ \text{ and } \ I^K_{{\color{RedViolet} P_L}}(\pm)\defeq f^{-1}(V_K(\pm)) \cap {\color{RedViolet} \overline{P_L} \cap \overline{(P+1)_L}}, $$ 
	then for all $K \in \Z^+$ and $P\in \Z$, 
	\begin{align*}
	\Gamma({\color{NavyBlue}(P+1)_R}{\color{RedViolet} K_L} {\color{NavyBlue}\cl{0_R}},\mp) &= \Gamma({\color{NavyBlue} P_R K_R \cl{0_R}}, \pm) &&\text{in } I^K_{{\color{NavyBlue} P_R}}(\pm) \\
	\Gamma({\color{RedViolet}(P+1)_L}{\color{NavyBlue} K_R \cl{0_R}}, \pm) &= \Gamma({\color{RedViolet}P_L}{\color{RedViolet} K_L} {\color{NavyBlue}\cl{0_R}},\mp) && \text{in } I^K_{{\color{RedViolet} P_L}}(\pm). 
	\end{align*}
	Any further identifications between canonical rays occur within the intersection of a fundamental domain and an itinerary domain, and hence can be expressed using \eqref{eq_further_overlap}. Moreover, arguing as for $z \mapsto \cosh(z)$, no two dynamic rays of $z \mapsto \cosh^2(z)$ land together, see the proof of Proposition~\ref{prop_landing_cosh}.
	
	\begin{figure}[htb]
	\centering
	\resizebox{0.6
		\textwidth}{!}{\begingroup%
		\makeatletter%
			\providecommand\color[2][]{%
				\errmessage{(Inkscape) Color is used for the text in Inkscape, but the package 'color.sty' is not loaded}%
				\renewcommand\color[2][]{}%
			}%
			\providecommand\transparent[1]{%
				\errmessage{(Inkscape) Transparency is used (non-zero) for the text in Inkscape, but the package 'transparent.sty' is not loaded}%
				\renewcommand\transparent[1]{}%
			}%
			\providecommand\rotatebox[2]{#2}%
			\newcommand*\fsize{\dimexpr\f@size pt\relax}%
			\newcommand*\lineheight[1]{\fontsize{\fsize}{#1\fsize}\selectfont}%
			\ifx\svgwidth\undefined%
			\setlength{\unitlength}{396.8503937bp}%
			\ifx\svgscale\undefined%
			\relax%
			\else%
			\setlength{\unitlength}{\unitlength * \real{\svgscale}}%
			\fi%
			\else%
			\setlength{\unitlength}{\svgwidth}%
			\fi%
			\global\let\svgwidth\undefined%
			\global\let\svgscale\undefined%
			\makeatother%
			\begin{picture}(1,0.71428571)%
			\setlength\tabcolsep{0pt}%
			\put(0,0){\includegraphics[width=\unitlength,page=1]{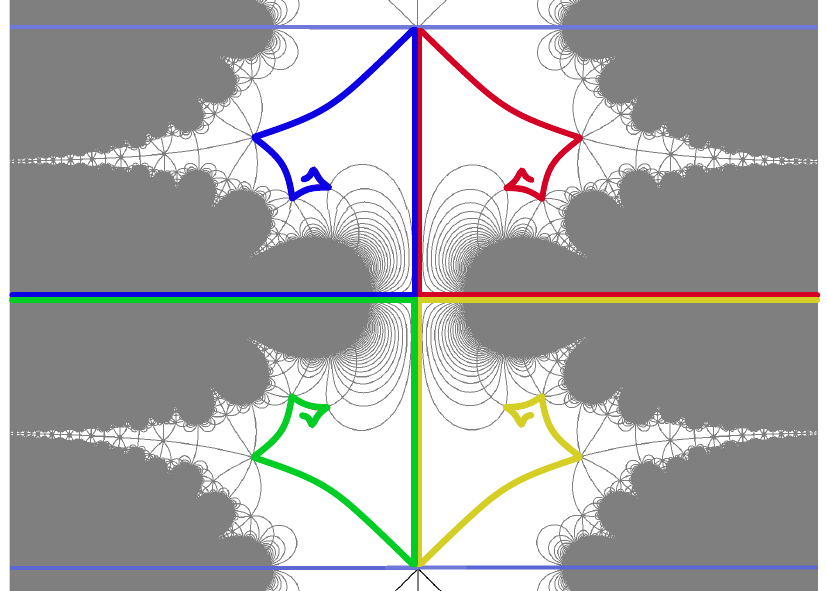}}%
			\end{picture}%
			\endgroup
		}	
		\caption[asd]{Picture showing four canonical tails of $z \mapsto \cosh^2(z)$ that contain the critical point $0$. These belong to the canonical rays $\Gamma({\color{RedViolet} 0_L} {\color{NavyBlue}\cl{0_R}},-)$ (in blue), $\Gamma({\color{NavyBlue}\cl{0_R}}, + )$ (in red), $\Gamma({\color{RedViolet} 0_L} {\color{NavyBlue}\cl{0_R}},+)$ (in green) and $\Gamma({\color{NavyBlue}\cl{0_R}}, -)$ (in yellow).}
		\label{fig:rays_cosh222}
	\end{figure}
	
\end{example}

\appendix
\section{Itineraries and rays landing together}\label{sec_landing}

This section concerns certain sps maps in $\B$ that have dynamic rays on their Julia sets. Namely, in \cite{mio_newCB}, the following class is introduced:
\begin{defn} We say that $f \in \B$ belongs to the \textit{class $\CB$} if $J(\lambda f)$ is a Cantor bouquet for some $\vert \lambda \vert$ sufficiently small.
\end{defn}
It follows from \cite[Theorem 1.4]{mio_newCB} that some iterate of each escaping point of any function $f\in \CB$ can be joined to infinity by a dynamic ray, and that the concept of canonical rays can be extended to maps in this class. Moreover, the class $\CB$ includes all finite compositions of class $\B$ functions of finite order, and in particular, all cosine maps; see \cite[Proposition 6]{mio_newCB}.
 
In this section, we extend the concept of \textit{itineraries} introduced in \S\ref{sec_cosh_indent} for the functions $z\mapsto \cosh(z)$ and $z \mapsto \cosh^2(z)$ to all strongly postcritically separated functions in class $\CB$, and use it to provide some combinatorial criteria for their canonical rays landing together in Theorem \ref{thm_landing_intro}. This idea, that comes from polynomial dynamics, has already been used in the study of the exponential and cosine families; see for example \cite{lasseTopExp, dierkParadox}. Moreover, a more general and systematic definition of itineraries for geometrically finite functions can be found in \cite[Chapter~5]{helena_thesis}. We say that $f\in \B$ is \textit{geometrically finite} if $S(f)\cap F(f)$ is compact and $P_J$ is finite. In particular, only attracting and parabolic basins can occur as Fatou components for maps in this class; \cite[Proposition 2.5]{helena_landing}. Since for strongly postcritically separated functions in class $\B$ the only possible Fatou components are attracting basins, \cite[Lemma 2.6]{mio_orbifolds}, some of the definitions in \cite[Chapter 5]{helena_thesis} will adapt to our setting. 

For the rest of the section, let us fix $f\in \CB$ and strongly postcritically separated.
We note that the definition of signed addresses that we provided for sps cosine maps in \eqref{eq_signedadd}, as well as that of canonical rays, extends to  all sps maps in $\CB$, see \cite[\S3]{mio_signed_addr}. Let us also fix $g\defeq \lambda f$ of disjoint type for some $\lambda \in \C^\ast$, and let $J(g)_\pm\defeq J(g)\times \{-,+\}.$ For each $\s\in \Addr(g)$ and $\ast \in \{-,+\}$, denote $J_{(\s, \ast)}\defeq J_\s\times \{\ast\}$. Let us define the function $\tilde{g}\colon J(g)_\pm\rightarrow J(g)_\pm$ by $\tilde{g}(z, \ast)\defeq(g(z), \ast).$ Then, \cite[Theorem 6.5]{mio_splitting} states that there exists a continuous surjective map 
\begin{equation}\label{eq_phiCB}
\phi: J(g)_\pm \to J(f)
\end{equation}
such that 
\begin{equation}\label{eq_com_CB}
	f\circ\phi = \phi\circ \tilde{g}.
\end{equation}
Moreover, for each  $(\s, \ast) \in \Addr(g)_\pm$, $\phi \colon J_{(\s, \ast)} \rightarrow \overline{\Gamma(\s, \ast)}$ is a bijection and  $\overline{\Gamma(\s, \ast)}$ is a canonical ray together with its endpoint. In particular, $\phi$  establishes a one-to-one correspondence between $\Addr(g)_\pm$ and $\Addr(f)_\pm$; see \cite[Observation~6.6]{mio_splitting}. 

Our first goal is to define a forward-invariant closed set $K \Supset P(f)$ such that $J(f)\subset \cl{\C \setminus K}$ and so that each connected component of $\C \setminus K$ is simply-connected. We will then define itineraries for $f$ using such components. More specifically, we will define $K$ as a union of two sets $K_J$ and $K_F$, the first consisting of the union of all canonical rays whose endpoints are in $P_J$, and the second comprising all points in $P_F$. 
Recall from Definition~\ref{def_ray} that $E(f)$ denotes the set of endpoints of dynamic rays of $f$.
\begin{prop} [Set of rays sharing their endpoint is closed]\label{prop_endpoints} For each $z\in E(f)$, denote by $\mathcal{R}(z)$ the set of canonical rays that land at~$z$. Then, $\cl{\mathcal{R}}(z)\defeq\mathcal{R}(z) \cup \{z\}$ is closed. 
\end{prop}
\begin{proof} Let $\phi$ be the function from \eqref{eq_phiCB}. Since $\phi$ is continuous, $\phi^{-1}(z)$ is a closed set of $J(g)_\pm$, being each connected component of $\phi^{-1}(z)$ the endpoint $(e_\s, \ast)$ of a set $J_{(s, \ast)}\subset J(g)_\pm$. Let $J(g)_\pm \cup\{\widetilde{\infty}\}$ be the one-point compactification of $J(g)_\pm$; see \cite[Lemma~5.8]{mio_splitting}. Then, the set
	$\mathcal{P}\defeq \{J_{(\s, \ast)} \colon (e_\s, \ast)\in \phi^{-1}(z) \}\cup \{\widetilde{\infty}\}$
	is compact. We can extend $\phi$ to a continuous map $\phi': J(g)_{\pm} \cup \lbrace \widetilde{\infty} \rbrace \rightarrow J(f) \cup \lbrace \infty \rbrace$ by defining $\phi'(\widetilde{\infty})=\infty$, see the proof of \cite[Theorem 6.5]{mio_splitting}. By continuity of $\phi'$, we have that $\phi'\left(\mathcal{P} \right)$ is compact. By definition of $\phi'$, it must be the case that $\phi'(\mathcal{P}\setminus \{\widetilde{\infty}\})=\phi(\mathcal{P}\setminus \{\widetilde{\infty}\})$, and by removing $\lbrace\infty\rbrace$ from the codomain of $\phi'$, we can conclude that $\phi(\mathcal{P}\setminus \{\widetilde{\infty}\})=\overline{\mathcal{R}}(z)$ is (relatively) closed in $J(f)$ with respect to the original topologies.
\end{proof}

\begin{discussion}[Definition of the set $K_J$] Denote
	\begin{equation}
	K_J\defeq \bigcup_{z\in P_J \cap E(f)} \overline{\mathcal{R}}(z), 
	\end{equation}
	and observe that by Proposition \ref{prop_endpoints} and since $P_J$ is discrete, $K_J$ is closed. 
\end{discussion}

In addition, we wish to include in $K$ the compact set $P_F$. Note that $\C \setminus P_F$ is open but not necessarily simply-connected. The idea is to remove a full set $K_F$ such that $F(f) \supset K_F \supset P_F$, together with a collection of curves that connect each connected component of $K_F$ to infinity. A piece of any such curve will be a dynamic ray, and the other piece will be a preperiodic simple curve inside an attracting basin of $F(f)$. For each attracting periodic point $z_0 \in F(f)$, $\mathcal{A}^\ast(z_0)$ denotes the immediate attracting basin of $z_0$. 
\begin{defn}[Attracting periodic rays~{\cite[Definition 5.2]{helena_thesis}}] Let $f$ be a transcendental entire function and let $z_0$ be an attracting periodic point of $f$ of period $n$. A simple curve $\alpha \colon (0,\infty) \rightarrow \mathcal{A}^\ast(z_0)$ is called an \textit{attracting periodic ray} of $f$ at $z_0$ (of period $n$) if
	\begin{itemize}
		\item[(i)] $f^n(\alpha(t)) = \alpha(2t)$,
		\item[(ii)] $\lim_{t\rightarrow \infty} \alpha(t) = z_0,$
		\item[(iii)] $\lim_{t\rightarrow 0} \alpha(t) = w$, where $w \in \partial \mathcal{A}^\ast(z_0)$ is a periodic point of $f$ of period $d\vert n$.
	\end{itemize}
\end{defn}
\begin{observation}[Images of attracting rays are attracting rays]\label{obs_attrray}
	If $\alpha$ is an attracting periodic ray of $f$ at $z_0$, then $f(\alpha)$ is an attracting periodic ray of $f$ at $f(z_0)$. Furthermore, if $f$ is strongly postcritically separated, by \cite[Lemma 2.6]{mio_orbifolds}, $w=\lim_{t\rightarrow 0} \alpha(t)$ must be a repelling periodic point of $f$.
\end{observation}
The next proposition tells us that we can find at least one attracting periodic ray for every attracting periodic point that also contains a prescribed point belonging to its immediate basin of attraction. This result is a version of \cite[Proposition 5.3]{helena_thesis} stated for our class of maps. Since the proof is exactly the same as for geometrically finite maps, we omit it.
\begin{prop}[Attracting rays with prescribed points]\label{prop_attr_ray}
Let $f\in \B$ be strongly postcritically separated and let $z_0$ be an attracting periodic point of $f$. Then, for any point $\xi$ that belongs to the unbounded component of $\mathcal{A}^\ast(z_0) \setminus P(f)$, there exists an attracting periodic ray of $f$ at $z_0$ in $\mathcal{A}^\ast(z_0) \setminus P(f) $ that contains $\xi$.
\end{prop}

Our next goal is to define $K_F$ by enclosing $P_F$ with a finite number of connected sets. Each of them will consist of a bounded domain in $F(f)$, together with (a preimage of) an attracting periodic ray, that either has an endpoint at infinity or at a repelling periodic point $p$. In the latter case, we will include $\overline{\mathcal{R}}(p)$ in $K_F$. We now formalize these ideas following the approach in \cite[Definition and Proposition 5.4]{helena_thesis}:
\begin{discussion}[Definition of the set $K_F$]\label{dis_KF}
	Since $P_F$ is compact, there exists a finite collection $\{A_i\}^n_{i=1}$ of connected components of $F(f)$ such that $\bigcup^n_{i=1}A_i \supset P_F$. We can assume that $\{A_i\}^n_{i=1}$ is minimal in the sense that $P_F \cap A_i \neq \emptyset$ for all $1\leq i\leq n$. We moreover can assume without loss of generality that $f$ has only one attracting cycle, since otherwise the same argument applies to each cycle. Let $\{z_1,\ldots, z_m\}$ be this attracting cycle and let $\{A_i\}^m_{i=1}$ be the corresponding immediate basins, labelled so that $A_i \ni z_i$ for all $1\leq i\leq m$, and so that $f(A_j)\subset A_{j-1}$ for all $2\leq j\leq n$. Let us pick any point $z^\ast \in A_n$ so that $f^j(z^\ast)$ does not belong to $P(f)$ for every $0\leq j\leq n-m$. We can find a collection of $n$ bounded Jordan domains $\{J_i\}^n_{i=1}$ such that $J_i \subset A_i$, $J \defeq \bigcup^n_{i=1} J_i \supset (P_F \cup \Orb^+(z^\ast))$ and $f(J)\Subset J \Subset F(f)$, see \cite[Proposition 3.1]{helena_thesis}. Let $\xi \defeq f^{n-m}(z^\ast) \in J_m \setminus P_F$. By Proposition \ref{prop_attr_ray}, there exists an attracting periodic ray $\alpha_m$ at $z_m \in J_m$ that contains $\xi$. For each $1 \leq i < m$, we denote the iterated forward image of $\alpha_m$ by $\alpha_i \defeq f^{m-i}(\alpha_m)$, which by Observation \ref{obs_attrray} is an attracting periodic ray at $z_i$. Note that it might occur that $\lim_{t \rightarrow 0} \alpha_i(t) = \lim_{t \rightarrow 0} \alpha_k(t)$ for some $i \neq k$.
	
	Let $\widetilde{\alpha}_m$ be the piece of $\alpha_m$ that connects $\xi$ to $\partial A_m$, let $w_m\defeq \lim_{t\rightarrow 0} \alpha_m(t)$ be the periodic endpoint of $\alpha_m$, and define the curve $\alpha_{m+1}$ as the connected component of $f^{-1}(\widetilde{\alpha}_m)$ belonging to $A_{m+1}$ that contains the point $f^{n-m-1}(z^\ast)$. By Proposition \ref{prop_attr_ray}, since $\alpha_m\cap P(f)=\emptyset$, this curve is unique and well-defined, and moreover, $ \lim_{t \rightarrow 0} \alpha_{m+1}(t)$ is either $\infty$ when $w_m$ is an asymptotic value, or it is a point $w_{m+1}$ such that $f(w_{m+1})=w_m$. Similarly and proceeding recursively, we define for each $m+2\leq j \leq n$ the curve $\alpha_j$ as the connected component of $f^{-1}(\alpha_{j-1}^\ast)$ belonging to $A_{j}$ that contains $f^{n-j}(z^\ast)$, and that in particular has analogous properties to those of the curve $\alpha_{m+1}$. That is, either $ \lim_{t \rightarrow 0} \alpha_{j}(t)$ is infinity, or it is a point $w_{j}$ such that $f(w_{j})=w_{j-1}$. Note that for $m+1\leq j\leq n$, the union $J_j \cup \alpha_j$ is a connected set, since $f^{n-j}(z^\ast) \in \alpha_j \cap J_j$.
	
	For every $1\leq i \leq n$ we define $\tilde{K}_i \defeq J_i \cup\alpha_i$, which by construction is closed, connected, $f(\tilde{K}_i ) \subset\tilde{K}_{i-1}$ for all $i \geq 2$ and $f(\tilde{K}_1) \subset \tilde{K}_m.$ Note that the sets $\{\tilde{K}_i\}_i$ are not necessarily simply-connected, as the curve $\alpha_i$ might intersect $\partial J_i$ more than twice. Thus, we define $\tilde{K}$ as the fill-in of $\bigcup_i\tilde{K}_i$; that is, $\tilde{K}$ equals $\bigcup_i\tilde{K}_i$ together with all bounded components of $\C \setminus \bigcup_i\tilde{K}_i $. Then, $\tilde{K}$ is a closed, connected and simply-connected set that by the Open Mapping Theorem satisfies $f(\tilde{K}) \subset \tilde{K}$. Moreover, $\tilde{K}$ consists of finitely many connected components, and $(\C \setminus \tilde{K}) \cap P_F =\emptyset$.
	
	By construction and Observation \ref{obs_attrray}, each connected component of $\tilde{K}$ that intersects the attracting cycle $\{z_1,\ldots, z_m\}$ contains exactly one periodic point in its boundary. Namely, the (non-separating) endpoint of an attracting periodic ray, and in particular belongs to $J(f)$. Let us label the distinct points that arise as a finite limit $\lim_{t \rightarrow \infty} \alpha_i(t)$
	for some $1\leq i \leq n$ as $\{w_1, \ldots,w_l\}\eqdef W$, noting that it might occur that $l < n$. Every $w_i\in W$ is a (pre)periodic point in $J(f).$ Let $V$ be the minimal set that contains $W$ and satisfies $f(V) \subset V$, i.e., $V$ is the set of forward images of the points in $W$.
	We define
	$$ K_F\defeq \bigcup_{w\in V} \overline{\mathcal{R}}(w)\cup \tilde{K},$$ 
	and note that by Proposition \ref{prop_endpoints} and since $\tilde{K}$ is closed, $K_F$ is also closed. Observe also that each of the connected components of $\C \setminus K_F$ is a simply-connected domain. 
\end{discussion}
Finally, we define 
\begin{equation}
K \defeq K_J \cup K_F,
\end{equation}
which is closed as it is the union of two closed sets. Note that the sets $K_J$ and $K_F$ might share some (piece of) periodic ray in $\overline{\mathcal{R}}(w)$ for some $w\in V$. By construction, the set $K$ is forward-invariant, that is,
\begin{equation}
f(K)\subseteq K.
\end{equation}
Moreover, $\C \setminus K \cap P(f)~=~\emptyset$ and each connected component of $\C \setminus K$ is simply-connected, since otherwise $K$ would enclose a domain that escapes uniformly to infinity, contradicting that $\text{int}(I(f))=\emptyset$ as $f\in \mathcal{B}$, \cite{eremenkoclassB}. Thus, since $f$ is an open map, all connected components of $f^{-1}(\C\setminus K)$ are simply-connected, which we label as $U_0, U_1,\ldots$. We denote by $\mathcal{U}$ the set of all those components. With slight abuse of notation, $\mathcal{U}$ will sometimes also denote its union. 

In the next proposition we assign to each of the landing canonical rays $\overline{\Gamma(\s, \ast)}$ a unique component of $\mathcal{U}$; compare to Proposition \ref{prop_itin_cosh}.
\begin{prop}[Each canonical ray is in the closure of a unique $U\in \mathcal{U}$]\label{prop_U2} For each $(\s, \ast)\in \Addr(f)_\pm$,  either $\overline{\Gamma(\s, \ast)}$ is totally contained in $K$, or there exists a unique component  $U\in \mathcal{U}$ such that $\overline{\Gamma(\s, \ast)} \subset \cl{U}$. In the latter case we denote 
	$$U (\s, \ast)\defeq U.$$		
\end{prop}

\begin{proof}
	Let $\gamma$ be an unbounded subcurve of $\overline{\Gamma(\s, \ast)}$. Then, by construction, if $\gamma$ belongs to a connected component of $K$, then $\overline{\Gamma(\s, \ast)}$ belongs to $K$. Otherwise, $\gamma\subset U$ for some $U \in \mathcal{U}$ and $\overline{\Gamma(\s, \ast)} \subset \overline{U}$. This can be seen either using the definition of canonical rays as nested sequences of left or right extensions, see \cite[Theorem 3.8]{mio_signed_addr}. Alternatively, it is easy to check the analogous propert for the model space $J(g)_\pm$, and it can be transferred using the continuous map $\phi$ from \eqref{eq_phiCB}, that preserves the orders of $\Addr(g)_\pm$ and $\Addr(f)_\pm$, see \cite[Observation 6.6]{mio_splitting}.
\end{proof}

\begin{defn}[Itineraries for canonical rays]
	For each $(\s, \ast)\in \Addr(f)_\pm$ we define the \textit{itinerary} of $(\s, \ast)$ as the sequence
	$$\itin(\s, \ast)\defeq U(\s, \ast) U(\sigma(\s), \ast)U(\sigma^2(\s), \ast)\ldots, $$
	whenever it is defined. 
	We say that an itinerary is \textit{bounded} if only finitely many different elements of $\mathcal{U}$ occur in it.
\end{defn}

\begin{observation}[Itineraries of points]\label{obs_itin_sps} Since by Proposition \ref{prop_U2}, for each $(\s, \ast) \in \Addr(f)_\pm$, $f(\overline{\Gamma(\s, \ast)})\subset \overline{\Gamma}(\sigma(\s), \ast)$, if $z\in \overline{\Gamma(\s, \ast)}$ and $\itin(\s, \ast)=U_0U_1\ldots$, then $f^i(z)\subset U_i$ for all $i\geq 0$. 
\end{observation}

\begin{prop}\label{prop_bounded} If the itinerary of any $(\ul{\alpha}, \ast)\in \Addr(f)_\pm$ is bounded, then 
	the endpoint $z$ of $\Gamma(\ul{\alpha}, \ast)$ has bounded orbit, i.e., $\sup_{j\geq 0} \vert f^j(z)\vert <\infty.$ 
\end{prop}	
\begin{proof}
	First we claim that each itinerary component $U\in \mathcal{U}$ only intersects finitely many fundamental domains in an unbounded component. The reason being that  $S(f)\cap I(f)$ is finite as $f\in \CB$ is sps, and so $\partial U$ only contains finitely many connected components that separate the plane, namely, those containing an escaping critical point. Each of these components contains two ray tails that are totally contained in two different fundamental domains. Using this and that $U$ cannot contain accesses to infinity between tracts, it is easy to see that the claim holds. Hence, we have that if $\itin(\ul{\alpha}, \ast)$ is bounded, $\ul{\alpha}$ must also be \textit{bounded}, that is, only finitely many different fundamental domains occur in $\ul{\alpha}$. 
	
	Recall that the map $\phi$ from \eqref{eq_phiCB} establishes a one-to-one correspondence between $\Addr(g)_\pm$ and $\Addr(f)_\pm$; namely $\phi(J_{(\ul{\alpha}, \ast)})=\overline{\Gamma}(\ul{\alpha}, \ast)$. Since $g$ is of disjoint type, if $\ul{\alpha}$ is bounded, then the endpoint $e_{\ul{\alpha}}$ of $J_{\ul{\alpha}}$ has bounded orbit under the map $g$, see \cite[Proposition 3.10]{lasse_arclike}. In addition, see the proof of \cite[Theorem 6.5]{mio_splitting}, there exists a constant $M$, independent of $\ul{\alpha}$, such that  $\vert \phi(e_{\ul{\alpha}}, \ast) - e_{\ul{\alpha}} \vert \leq  M$. This together with \eqref{eq_com_CB} implies that the endpoint $\phi(e_{\ul{\alpha}}, \ast)$ of $\Gamma(\ul{\alpha}, \ast)$ also has bounded orbit.
\end{proof}

The following is the main result of this section. 
\begin{thm}[Criterion for rays landing together]\label{thm_landing_intro} Let $f\in \CB$ be strongly postcritically separated. Then two canonical rays with bounded itinerary land together if and only if they have the same itinerary.
\end{thm}

\begin{proof} If two canonical rays land together, then by Proposition \ref{prop_U2} and Observation \ref{obs_itin_sps}, they must have the same itinerary. For the other implication, let $\Gamma(\s, \ast)$ and $\Gamma(\ultau, \star)$ be two different canonical rays, that is, $(\s, \ast)\neq (\ultau, \star)$, and let $p_0\defeq p_{(\s, \ast)}$ and $q_0\defeq q_{(\ultau, \star)}$ be their respective endpoints. Moreover, for each $n\geq 0$, denote $p_n\defeq f^n(p_0)$ and $q_n\defeq f^n(q_0)$. 
	By assumption, $\itin(\s, \ast)=\itin(\ultau, \ast)=U_0U_1U_2\ldots$ is bounded, that is, there exists a finite collection $\mathcal{V}\defeq\{V_i\}_{i\in I}$ of domains in $\mathcal{U}$ such that $U_n\in \mathcal{V}$ for all $n\geq 0$. In particular, by Observation \ref{obs_itin_sps}, $q_n,p_n\in U_n$ for all $n\geq 0$. We want to show that $p_0=q_0$.
	
	By Proposition \ref{prop_bounded}, both $p_0$ and $q_0$ have bounded orbits, and hence, for each $V_i\in \V$, we can find a bounded simply-connected domain $W_i\in V_i$ such that $$(\Orb^+(p_0)\cup \Orb^+(q_0))\cap V_i \subset W_i.$$
	Moreover, since all domains in $\mathcal{U}$ have a locally connected boundary, $W_i$ can be chosen so that $\partial W_i$ is locally connected.
	
	Since $f\in \B$ and sps, \cite[Theorem 1.1]{mio_orbifolds} states that there exist hyperbolic orbifolds $\Ort=(\widetilde{S},\tilde{\nu})$ and $\Or=(S,\nu)$ such that $\tilde{S}\subset S\subseteq \C$, $f \colon\Ort\to\Or$ is an orbifold covering map, and there exists a constant $\Lambda > 1$ such that $\Vert D f(z)\Vert_{\Or}\defeq(\vert f'(z)\vert\rho_{\Or}(f(z)))/\rho_{\Or}(z) \geq~\Lambda$ for all $z\in \mathcal{U}$, where $\rho_{\Or}$ denotes the density of its orbifold metric. We moreover denote by $\ell_\Or$ and $d_\Or$ the corresponding orbifold length and distance; see \cite[\S 3]{mio_orbifolds} for definitions. In particular, we have defined $\mathcal{U}$ such that $\mathcal{U} \subset \tilde{S}\cap S$, compare \ref{dis_KF} with the proof of \cite[Definition and Proposition 5.1]{mio_orbifolds}. Since the sets in $\{W_i\}_{i \in I}$ do not contain postsingular points,  \cite[Theorem 7.5]{mio_orbifolds} implies that for each $i\in I$, there exists a constant $\mu_i$ such that if $\delta_n \subset W_i$ is a curve joining $p_n$ and $q_n$ for some $n\in \N$, then there exists a curve $\gamma_n \subset W_i$ that also has endpoints $p_n$ and $q_n$, that is homotopic to $\delta_n$ with respect to $P(f)$ and so that $\ell_\Or(\gamma_n)<\mu_i$. This means that there exists an inverse branch $F$ of $f^n$ such that $\gamma^n_0\defeq F(\gamma_n)$ joins $p_0$ and $q_0$; see \cite[\S7]{mio_orbifolds} for more details. In particular, if we let $\mu\defeq \max_{i \in I} \mu_i$, then for each $n\in \N$, $d_\Or(p_0,q_0) \leq \ell_\Or(\gamma^n_0) \leq \ell_\Or(\gamma_n)/\Lambda^n\leq \mu/\Lambda^n $, which tends to $0$ as $n\rightarrow \infty$, and thus $p_0=q_0$, as we wanted to show.
\end{proof}
\bibliographystyle{alpha}
\bibliography{biblioComplex}
\end{document}